\documentclass[11pt]{amsart}

\usepackage{graphicx}
\usepackage[utf8]{inputenc}
\usepackage[english]{babel}
\usepackage{csquotes}
\usepackage{amssymb}
\usepackage{tikz-cd}
\usepackage{bm}

\usepackage[backend=biber,style=alphabetic,sorting=nyt]{biblatex}

\usepackage{enumerate}
\usepackage{enumitem}
\usepackage{amsmath,amsfonts,pictex,graphicx,fullpage,etex,tikz,hyperref}
\newtheorem{theorem}{Theorem}[section]
\newtheorem{lemma}[theorem]{Lemma}
\newtheorem{proposition}[theorem]{Proposition}
\newtheorem{corollary}[theorem]{Corollary}

\theoremstyle{definition}
\newtheorem{definition}[theorem]{Definition}

\newtheorem{problem}[theorem]{Problem}

\theoremstyle{remark}
\newtheorem{remark}[theorem]{Remark}
\newtheorem{notation}[theorem]{Notation}
\newtheorem{claim}[theorem]{Claim}

\numberwithin{equation}{section}

\newcommand{\abs}[1]{\lvert#1\rvert}
\newcommand{\norm}[1]{\lVert #1 \rVert}

\newcommand{\T}{\mathrm{T}}
\newcommand{\M}{\mathrm{M}}
\newcommand{\R}{\mathbb{R}}
\newcommand{\Z}{\mathbb{Z}}
\newcommand{\G}{\mathcal{G}}
\newcommand{\x}{\vect{x}}
\newcommand{\y}{\vect{y}}
\newcommand{\I}{\mathrm{I}}
\newcommand{\N}{\mathbb{N}}
\newcommand{\SL}{\mathrm{SL}}
\newcommand{\GL}{\mathrm{GL}}
\newcommand{\dd}{\mathrm{d}}

\newcommand{\E}{\mathrm{E}}

\newcommand{\vect}[1]{\boldsymbol{\mathbf{#1}}}
\newcommand{\Span}{\operatorname{Span}}
\newcommand{\Ad}{\operatorname{Ad}}
\newcommand{\inv}{^{-1}}

\begin{document}

\title{Equidistribution of curves in homogeneous spaces and Dirichlet's approximation theorem for matrices}

\author{Nimish Shah}
\address{Department of Mathematics, The Ohio State University, Columbus, OH 43210}
\email{shah@math.osu.edu}

\author{Lei Yang}
\address{Mathematical Sciences Research Institute, Berkeley, CA, 94720}
\email{lyang@msri.org}

\curraddr{Einstein Institute of Mathematics, Hebrew University of Jerusalem, Jerusalem, 9190401, Israel}
\email{yang.lei@mail.huji.ac.il}

\thanks{N. Shah acknowledges support from an NSF grant.}
\thanks{L. Yang is supported in part by a Postdoctoral Fellowship at MSRI}

\date{}



\begin{abstract}
In this paper, we study an analytic curve $\varphi: I=[a,b]\rightarrow \mathrm{M}(m\times n, \R)$ 
in the space of $m$ by $n$ real matrices, and show that if $\varphi$ satisfies certain geometric condition, then 
for almost every point on the curve, the Diophantine approximation given by Dirichlet's Theorem can not be improved. To do this, 
we embed the curve into some homogeneous space $G/\Gamma$, and prove that under the action of some expanding diagonal subgroup $A= \{a(t): t \in \R\}$, the translates of the curve tend to be equidistributed in $G/\Gamma$, as $t \rightarrow +\infty$.
\end{abstract}

\maketitle

\section{Introduction}


\subsection{Diophantine approximation for matrices}
In 1842, Dirichlet proved the following result on simultaneous approximation of a matrix of real numbers by integral vectors:
{\em Given two positive integers $m$ and $n$, a matrix $\Phi \in \mathrm{M}(m\times n, \R)$, and any $N>0$,
there exist integral vectors $\vect{p} \in \Z^n\setminus\{\vect{0}\}$ and  $\vect{q} \in \Z^m$ such that 
\begin{equation} \label{eq:Dirichlet}
\|\vect{p}\| \leq N^m  \quad{and} \quad \norm{\Phi \vect{p} -\vect{q}} \leq N^{-n},
\end{equation}
where $\norm{\cdot}$ denotes the supremum norm; that is, $\norm{\x}:=\max_{1\leq i\leq k}\abs{x_i}$ for any
$\x =(x_1,x_2,\dots, x_k)\in \R^k$. }

\par Now we consider the following finer question: for a particular $m$ by $n$ matrix $\Phi$, could we improve Dirichlet's Theorem? By improving Dirichlet's 
Theorem, we mean there exists a constant $0 < \mu <1$, such that for all large $N > 0$, there exists nonzero integer vector $\vect{p} \in \Z^n$ with
$\|\vect{p}\| \leq \mu N^m$, and integer vector $\vect{q} \in \Z^m$ such that $\|\Phi \vect{p} - \vect{q}\| \leq \mu N^{-n}$. 
If such constant
$\mu$ exists, then we say $\Phi$ is $DT_{\mu}$-improvable. And if $\Phi$ is $DT_{\mu}$-improvable for some $0<\mu <1$, then we say $\Phi$ is 
$DT$-improvable (here $DT$ stands for Dirichlet's Theorem). 
\par This problem was firstly studied by Davenport and Schmidt in \cite{Daven_Schm}, 
in which they proved that almost every matrix $\Phi \in \M(m\times n,\R)$ is not $DT$-improvable. In \cite{Daven_Schm}, they also proved the 
following result. For $m=1$ and $n=2$, $\M(1\times 2, \R) = \R^2$, one considers the curve $\phi(s)=(s,s^2)$ in $\R^2$. Then for almost every $s\in \R$
with respect to the Lebesgue measure on $\R$, $\phi(s)$ is not $DT_{1/4}$ improvable. This result was generalized by Baker in \cite{Baker}: for any 
smooth curve in $\R^2$ satisfying some curvature condition, almost every point on the curve is not $DT_{\mu}$ improvable for some $0<\mu <1$ depending on the curve. Bugeaud \cite{Bugeaud} generalized the result of Davenport and Schmidt in the following sense: for $m=1$, and general 
$n$, almost every point on the curve $\varphi(s)=(s,s^2,\dots, s^n)$ is not $DT_{\mu}$-improvable for some small constant $0<\mu<1$. Their proofs 
are based on the technique of regular systems introduced in \cite{Daven_Schm}.
\par Recently, based on an observation of Dani \cite{Dani}, as well as Kleinbock and Margulis \cite{Klein_Mar},
Kleinbock and Weiss \cite{Klein_Weiss} studied this 
Diophantine approximation problem in the language of homogeneous dynamics, and proved the following result: for $m=1$ and arbitrary $n$, if 
an analytic curve in $\M(1\times n,\R) \cong \R^n$ satisfies some non-degeneracy condition, then almost every 
point on the curve is not $DT_{\mu}$-improvable for some small constant $0<\mu<1$ depending on the curve. Based 
on the same correspondence, Nimish Shah \cite{Shah_2} proved the following stronger result: for $m=1$ and general $n$, if an analytic curve 
$\varphi: I=[a,b] \rightarrow \R^n$ is not contained in a proper affine subspace, then almost every point on the curve is not $DT$-improvable. For $m=n$, Lei Yang \cite{Yang_1} provided a geometric condition and proved that if an analytic curve 
$\varphi: I =[a,b] \rightarrow \M(n\times n, \R)$ satisfies the condition, then almost every point on $\varphi$ is not 
$DT$-improvable. The geometric condition given there provides some hint on solving the problem for general $(m,n)$, and will be discussed in detail later.
\par The purpose of this paper is to give a geometric condition for each $(m,n)$, and show that if an analytic curve 
$$\varphi: I=[a,b] \rightarrow \M(m\times n, \R)$$
satisfies the condition, then almost every point on $\varphi$ is not 
$DT$-improvable.
\par The geometric conditions called {\em generic} condition and {\em supergeneric} condition are defined as follows:
\begin{definition}
\label{def:geometric_condition}
\par $\quad$
\par For any $m$ and $n$, let 
$$\varphi: I=[a,b]\rightarrow \M(m\times n, \R)$$
denote an analytic curve.
\par For $m=n$, we say $\varphi$
is {\em generic\/} at $s_0\in I$ if there exists a subinterval $J_{s_0} \subset I$ such that for any $s \in J_{s_0}$, $\varphi(s) - \varphi(s_0)$ is invertible. 

\par In order to define {\em supergeneric} condition, we need additional notation. 

\par We consider the following two embeddings from $\M(m\times m,\R)$
to the Lie algebra $\mathfrak{sl}(2m, \R)$ of $\SL(2m,\R)$:
$$
\mathfrak{n}^{+}: X \in  \M(m\times m,\R) \mapsto \begin{bmatrix}\vect{0} & X \\ & \vect{0}\end{bmatrix} \in \mathfrak{sl}(2m, \R),
$$
and
$$
\mathfrak{n}^{-}: X \in  \M(m\times m,\R) \mapsto \begin{bmatrix}\vect{0} &  \\ X & \vect{0}\end{bmatrix} \in \mathfrak{sl}(2m, \R).
$$
Let
\begin{equation} \label{eq:E}
\mathcal{E}=\begin{bmatrix}\I_m & \\ & - \I_m\end{bmatrix} \in \mathfrak{sl}(2m,\R).
\end{equation}

A Lie subgroup $L$ of $H=\SL(2m, \R)$ is called {\em observable\/} if there exists a finite dimensional linear representation $V$ of 
$H$ and a nonzero vector $v \in V$ such that the subgroup of $H$ stabilizing $v$ is equal to $L$. A Lie subalgebra $\mathfrak{l}$ is called an {\em observable\/} Lie subalgebra of $\mathfrak{sl}(2m,\R)$, if it is the Lie algebra of some observable Lie subgroup $L \subset H$.

For the case of $m=n$, the curve $\varphi$ is called {\em supergeneric\/} at $s_0 \in I$ if it is {\em generic} at $s_0$ (with subinterval $J_{s_0} \subset I$), and for any proper observable subalgebra $\mathfrak{l}$ of $\mathfrak{sl}(2m,\R)$ containing $\mathcal{E}$, we have 
\begin{equation} \label{eq:supergeneric}
\{ \mathfrak{n}^{-}((\varphi(s_1)-\varphi(s_0))^{-1} -(\varphi(s_2)-\varphi(s_0))^{-1}): s_1, s_2 \in J_{s_0}\}\not\subset \mathfrak{l}.
\end{equation} 

\par For $m >n$, $\varphi$ is called {\em generic\/} ({\em supergeneric\/}) if its transpose 
$$\varphi^{\T}: I=[a,b] \rightarrow \M(n\times m, \R)$$
is {\em generic\/} ({\em supergeneric\/}).

\par For $m< n$, we express $\varphi(s)=[\varphi_{1}(s) ; \varphi_{2}(s)]$,
where $\varphi_1(s)$ is the first $m$ by $m$ block, and $\varphi_2(s)$ is the rest $m$ by $n-m$ block. We say $\varphi$ is {\em generic\/} ({\em supergeneric\/}) at $s_0\in I$, if there exists a subinterval $J_{s_0} \in I$ such that for any $s \in J_{s_0}$, $\varphi_1(s) - \varphi_1(s_0)$ is invertible; and if we define  $\psi: J_{s_0} \rightarrow \M(m \times (n-m), \R)$, by
\[
\psi(s) =  (\varphi_1(s)-\varphi_1(s_0))^{-1} (\varphi_2(s)-\varphi_2(s_0)) 
\]
then $\psi$ is {\em generic\/} ({\em supergeneric\/}) at some $s_1 \in J_{s_0}$. 

\par We say that $\varphi$ is {\em generic} ({\em supergeneric}) or satisfies {\em generic} ({\em supergeneric}) condition, if $\varphi$ is {\em generic} ({\em supergeneric}) at some $s_0\in I$. Since $\varphi$ is analytic, if it is {\em generic} ({\em supergeneric}) at one point of $I$ then it will be {\em generic} ({\em supergeneric}) at all but finitely many points of $I$. 
\end{definition}


\begin{remark} 
$\quad$
\begin{enumerate}
\item In \cite{Yang_1}, it is proved that for $m=n$, if there exists $s_0 \in I$ and a subinterval $J_{s_0} \subset I$ such that the derivative $\varphi^{(1)}(s_0)$ is invertible, $\varphi(s) - \varphi(s_0)$ is invertible for any $s \in J_{s_0}$, and $\{(\varphi(s)- \varphi(s_0))^{-1}: s \in J_{s_0}\}$ is not contained in any proper affine subspace of $\M(m\times m, \R)$, then almost every point on the curve is not 
$DT$-improvable. Here the {\em supergeneric\/} condition is weaker than the condition needed in \cite{Yang_1}.

\item In the case of $m$ and $n$ being co-prime, the {\em generic} condition directly implies the {\em supergeneric} condition. 

\item For $m=1$ and general $n$, the genericness condition is equivalent to the condition that the curve is not contained in a proper affine subspace of $\R^n$.

\end{enumerate}
\end{remark}

In this paper we will prove the following result:
\begin{theorem}
\label{thm_diophantine_part}
For any $m$ and $n$, if an analytic curve 
$$\varphi: I=[a,b] \rightarrow \M(m\times n ,\R)$$
is {\em supergeneric\/}, then almost every point on $\varphi$ is not 
$DT$-improvable. If $(m,n)=1$, then the same result holds for {\em generic\/} analytic curves.
\end{theorem}


\subsection{Equidistribution of expanding curves on homogeneous spaces}
\par Now we describe the correspondence between Diophantine approximation and homogeneous dynamics as follows.
\par Let $G= \mathrm{SL}(m+n,\R)$, and let $\Gamma = \mathrm{SL}(m+n,\Z)$. The homogeneous space $G/\Gamma$ can be identified with the space of unimodular lattices of $\R^{m+n}$. 
Every point $g\Gamma$ corresponds to the unimodular lattice $g\Z^{m+n}$. For $r>0$, let $B_{r}$ denote the 
ball in $\R^{m+n}$ centered at the origin and of radius $r$. For any $0<\mu<1$, the subset 
$$K_{\mu}:= \{\Lambda \in G/\Gamma: \Lambda\cap B_{\mu} = \{\vect{0}\}\}$$
contains an open neighborhood of $\Z^{m+n}$ in $G/\Gamma$.
Let us define the diagonal subgroup $A=\{a(t): t \in \R\}$ by
\begin{equation} \label{eq:a(t)}
a(t):= \begin{bmatrix}e^{nt} \I_m & \\ & e^{-mt}\I_n \end{bmatrix}.
\end{equation}

Now we consider the embedding 
              $$u: \M(m\times n, \R) \rightarrow \mathrm{SL}(m+n,\R)$$
              sending $\Phi \in \M(m\times n , \R)$ to 
              \begin{equation} \label{eq:u(Phi)} 
              u(\Phi):= \begin{bmatrix}\I_m & \Phi \\ & \I_n \end{bmatrix}.
              \end{equation}

Suppose for some $0<\mu<1$, and any $N>0$ large enough, there exist nonzero integer vector $\vect{p} \in \Z^n$ and integer 
   vector $\vect{q}\in \Z^m$ such that $\|\vect{p}\| \leq \mu N^m$ and $\|\Phi \vect{p} -\vect{q}\| \leq \mu N^{-n}$. 
   Then the lattice $a(\log N)u(\Phi)\Z^{m+n}$ has a vector $a(\log N)u(\Phi)(-\vect{q}, \vect{p})$ whose norm is less than $\mu$, i.e.,
   $a(\log N) u(\Phi) \Z^{m+n} \not\in K_{\mu}$ for all $N>0$ large enough. Thus, to show that $\Phi \in \M(m\times n, \R)$ is not 
   $DT_{\mu}$-improvable,
   it suffices to show that the trajectory $\{a(t)u(\Phi)[e]: t >0\}$ meets $K_{\mu}$ infinitely many times. In particular, for an analytic curve 
   $$\varphi: I=[a,b]\rightarrow \M(m\times n, \R),$$ if we could show that for almost every point $\varphi(s)$ on the curve the trajectory 
   \begin{equation}\label{equ_ae_orbit_dense}\{a(t)u(\varphi(s))[e]: t >0\} \text{ is dense, } \end{equation}
   then we could conclude that almost every $\varphi(s)$ is not $DT$-improvable. In particular, if one can prove that the 
   expanding curves $a(t)u(\varphi(I))[e]$ tend to be equidistributed in $G/\Gamma$ as $t \rightarrow +\infty$, then (\ref{equ_ae_orbit_dense}) will follow (see \cite{Shah_2} for detailed proof).
   It turns out that the equidistribution result described above holds for a much more general setting. In fact, one could prove the following result:
\begin{theorem}
\label{goal_thm}
Let $G$ be a Lie group containing $H = \mathrm{SL}(m+n,\R)$, and $\Gamma < G$ be a lattice of $G$. Let $\mu_G$ denote the unique $G$-invariant probability 
 measure on the homogeneous space $G/\Gamma$. Take $x = g\Gamma \in G/\Gamma$ such that its $H$-orbit $Hx$ is dense in $G/\Gamma$. Let us fix the 
 diagonal group
 $$A = \left\{a(t) = \begin{bmatrix}e^{nt}\I_m & \\ & e^{-mt}\I_n\end{bmatrix}\right\}.$$
 Let $\varphi: I=[a,b]\rightarrow \M(m\times n, \R)$ be an analytic curve. We embed the curve into $H$ by 
 $$u: X \in \M(m\times n , \R) \mapsto u(X) = \begin{bmatrix}\I_m & X \\ & \I_n\end{bmatrix}.$$
 For $t >0$, let $\mu_t$ denote the normalized parameteric measure on the curve $a(t)u(\varphi(I))x \subset G/\Gamma$; that is, for a 
 compactly supported continuous function $f \in C_c(G/\Gamma)$,
 $$\int f \dd\mu_t := \frac{1}{\abs{I}}\int_{s\in I} f(a(t)u(\varphi(s))x)\dd s.$$
 \par If $\varphi$ is {\em generic\/}, then every weak-$\ast$ limit measure $\mu_{\infty}$ of $\{\mu_t: t >0\}$ is still a 
probability measure.
 If the curve $\varphi$ is {\em supergeneric\/}, 
 then $\mu_t \rightarrow \mu_G$ as $t \rightarrow +\infty$ in weak-$\ast$ topology; that is, for any function $f\in C_c(G/\Gamma)$,
 $$\lim_{t\rightarrow +\infty} \frac{1}{\abs{I}} \int_{s \in I} f(a(t)u(\varphi(s))x)\dd s = \int_{G/\Gamma} f \dd \mu_G.$$
Moreover, if $(m,n)=1$, then {\em generic\/} property will imply that $\mu_t \rightarrow \mu_G$ as $t \rightarrow +\infty$.
\end{theorem}
\begin{remark} 
$\quad$
\begin{enumerate}
\item To prove Theorem \ref{thm_diophantine_part}, we only need the above theorem with $G=H=\mathrm{SL}(n+m,\R)$, $\Gamma = \mathrm{SL}(m+n,\Z)$,
 and $x = [e] = \Z^{m+n} \in G/\Gamma$.
 \item Even in the case $G=H=\mathrm{SL}(m+n,\R)$, Theorem \ref{goal_thm} is still much stronger than Theorem 
 \ref{thm_diophantine_part}, since it applies for arbitrary lattice $\Gamma < G$.
 \end{enumerate}
\end{remark}
\par The study of limit distributions of evolution of curves translated by diagonalizable subgroups in homogeneous spaces has its own interest and has a lot of interesting connections to geometry and Diophantine approximation. One could summarize this type of problems as follows:
\begin{problem}
\label{problem_equidistribution}
Let $H$ be a semisimple Lie group, generated by its unipotent subgroups. Fix a diagonalizable one parameter subgroup $A =\{a(t): t \in \R\} \subset H$, and let $U^{+}(A)$ denote the expanding horospherical subgroup of $A$ in $H$. Let $G$ be a Lie group containing $H$, and let $\Gamma$ be a lattice of $G$.
\par Let
      $$\phi: I=[a,b] \rightarrow H$$
      be a piece of analytic curve in $H$ with nonzero projection on $U^{+}(A)$ (this will make sure that the translates of $\phi(I)$ by $\{a(t): t >0\}$ expand). Given a point $x = g\Gamma \in G/\Gamma$, Ratner's Theorem tells that the closure of $Hx$ is a finite volume homogeneous subspace $Fx$, where $F$ is a Lie subgroup of $G$ containing $H$. Let $\mu_F$ denote the unique probability $F$-invariant measure supported on $Fx$.
      One can ask whether the expanding curves $\{a(t)\phi(I)x: t >0 \}$ tend to be equidistributed in $Fx$, i.e., as $t \rightarrow +\infty$, the normalized parametric measure supported on $a(t)\phi(I)x$ approaches $\mu_F$ in weak-$\ast$ topology.  
\end{problem}
\begin{remark}
Without loss of generality, in this paper, we always assume that $Hx$ is dense in $G/\Gamma$. If $Hx$ is not dense, suppose its closure is $Fx$, then we may replace $G$ by $F$, $\Gamma$ by $F\cap x \Gamma x^{-1}$ (which is a lattice of $F$ by the closeness of $Fx$). 
\end{remark}
\par Nimish Shah \cite{Shah_1} and \cite{Shah_3} studied the case
$H =\mathrm{SO}(n,1)$ and $G=\mathrm{SO}(m,1)$ where $m\geq n$. In this case
the diagonalizable subgroup $A=\{a(t): t \in \R\}$ is a fixed maximal $\R$-split Cartan subgroup of $H$. In \cite{Shah_1} it is proved that given an analytic curve 
$$\phi: I=[a,b] \rightarrow H,$$
and a point $x = g\Gamma \in G/\Gamma$, unless the natural visual map 
$$\mathrm{Vis}: \mathrm{SO}(n,1)/\mathrm{SO}(n-1) \cong \T^1(\mathbb{H}^n) \rightarrow \partial \mathbb{H}^n \cong \mathbb{S}^{n-1}$$
sends the curve $\phi(I)$ to a proper subsphere of $\mathbb{S}^{n-1}$, the translates $a(t)\phi(I)x$ of $\phi(I)x$ will tend to be equidistributed as $t \to +\infty$. In \cite{Shah_3}, the same result is proved when $\phi$ is only $C^n$ differentiable. In \cite{Shah_1} and \cite{Shah_3}, the obstruction of equidistribution is discussed and possible limit measures are described when the equidistribution fails. This result was generalized by Yang \cite{Yang_2} in the following sense: for $H=\mathrm{SO}(n,1)$ and arbitrary Lie group $G$ containing $H$, if the same condition on the curve holds, then the expanding curve $a(t)\phi(I)x$ tends to be equidistributed as $t \to +\infty$. Shah \cite{Shah_2} studied the case $m=1$ of the problem we consider in this paper, and proved that if the analytic curve $\varphi: I \rightarrow \M(1\times n, \R) = \R^n$ is not contained in a proper affine subspace of $\R^n$, then the equidistribution holds. It turns out that this condition is the same as {\em generic\/} condition for $m=1$. Later Yang \cite{Yang_1} studied the case $m=n$.
\par When the {\em generic\/} condition holds but {\em supergeneric\/} condition does not, we want to understand the obstruction of equidistribution and describe the limit measures of $\{\mu_t: t >0\}$ to some extent. This requires more subtle argument. In \cite{Shah_1} and \cite{Shah_1}, obstruction of equiditribution and description of limit measures are clearly given unconditionally for the case $H=\mathrm{SO}(n,1)$ and $G = \mathrm{SO}(m,1)$ in the set up of Problem \ref{problem_equidistribution}. In our case, the problem becomes much more complicated. In this paper, we only discuss the case $n=km$, and we conjecture that similar result remains true for general $(m,n)$ such that $(m,n)>1$ (for the case $(m,n)=1$, {\em generic\/} is the same as {\em supergeneric\/}, so there is nothing in between).

\subsection{Relation to extremity of submanifolds in homogeneous spaces}
\par Another direction to study Diophantine properties of a real matrix $\Phi \in \M(m\times n, \R)$ is to determine whether $\Phi$ is 
very well approximable. We say $\Phi \in \M(m\times n ,\R)$ is very well approximable if there exists some constant $\delta>0$ such that there
exist infinitely many nonzero integer vectors $\vect{p} \in \Z^n$ and integer vectors $\vect{q}\in \Z^m$ such that
$$\norm{\Phi \vect{p} - \vect{q}} \leq \norm{\vect{p}}^{-n/m-\delta}.$$
A submanifold $\mathcal{U} \subset \M(m\times n,\R)$ is called extremal if with respect to the 
Lebesgue measure on $\mathcal{U}$, almost every point is not very well approximable. Based on the same correspondence due to Dani \cite{Dani} and due to Kleinbock and Margulis \cite{Klein_Mar}, this problem can also be studied
through homogenous dynamics. Kleinbock and Margulis \cite{Klein_Mar} proved that if a submanifold $\mathcal{U} \subset \M(1\times n ,\R)$ 
is nondegenerate, then $\mathcal{U}$ is extremal.  Kleinbock, Margulis and Wang \cite{Klein_Mar_Wang} later gave a necessary 
and sufficient condition of a submanifold of $\M(m\times n, \R)$ being extremal. The condition is stated in terms of a particular representation of $H = \SL(m+n, \R)$ and could not be translated to a geometric condition. Recently, Aka, Breuillard, Rosenzweig and de Saxc\'{e} \cite{ABRS} gave a family of 
subvarieties of $\M(m\times n, \R)$ called pencils, and announced a theorem stating that if a submanifold $\mathcal{U} \subset \M(m\times n , \R)$ is 
not contained in a pencil, then $\mathcal{U}$ is extremal. It turns out that the {\em generic\/} condition implies the condition given in \cite{ABRS}. We will discuss it in detail in Appendix 
\ref{App:AppendixA}.

\subsection{Organization of the paper}
\par The paper is organized as follows: In \S \ref{non-divergence}, assuming the {\em generic\/} condition on $\varphi$, we will relate a unipotent invariance to limit measures of $\{\mu_t: t>0\}$, and show that every limit measure is still a probability measure. This allows us to apply Ratner's theorem. In \S \ref{ratner_linearization} We will apply Ratner's theorem and the linearization technique to study the limit measure via a particular linear representation of $H$. Finally we will get a linear algebraic condition on $\varphi$. In \S \ref{lemmas}, we will recall and prove some basic lemmas on linear representations, which are essential in our proof. In \S \ref{proof_equidistribution}, assuming the {\em supergeneric\/} condition, we will give the proof of Theorem \ref{goal_thm}, as we have discussed before, Theorem \ref{thm_diophantine_part} will follow from Theorem
\ref{goal_thm}. In \S \ref{obstruction_equidistribution}, assuming the {\em generic\/} condition, we will study the obstruction of equidistribution and limit measures of $\{\mu_t: t >0\}$. We will only discuss the case $n=km$, and give a conjecture on general case. In the appendix, we will discuss the condition given in \cite{ABRS} and its relation to our {\em generic\/} condition.

\begin{notation}
 In this paper, we will use the following notation. \par For $\epsilon > 0$ small, 
 and quantities $A$ and $B$, $A \overset{\epsilon}{\approx} B$ means that 
 $\abs{A-B} \leq \epsilon$. Fix a right $G$-invariant metric $d(\cdot, \cdot)$ on $G$, then for $x_1, x_2 \in G/\Gamma$,
 and $\epsilon >0$, $x_1 \overset{\epsilon}{\approx} x_2$ means $x_2 =g x_1$ such that $d(g,e)< \epsilon$.
 \par For two related variable quantanties $A$ and $B$, $A \ll B$ means there exists a constant $C >0$ such that $A \leq CB$, and  
 $A \gg B$ means $B \ll A$. $O(A)$ denotes some quantity $\ll A$ or some vector whose norm is $\ll A$.
\end{notation}


\subsection*{Acknowledgement} The second author thanks The Ohio State University for hospitality during his visit when the project was initiated. Both authors thank MSRI where they collaborated on this work in Spring 2015. The second author thanks Dmitry Kleinbock for helpful discussions on the generic condition and for drawing his attention to the work of Aka, Breuillard, Rosenzweig and de Saxc\'{e} \cite{ABRS}.


\section{Non-divergence of the limit measures and unipotent invariance}
\label{non-divergence}

\subsection{Preliminaries on Lie group structures}
We first recall some basic facts of the group $H= \mathrm{SL}(m+n,\R)$.
Without loss of generality, throughout this paper we always assume $m\leq n$.
\par 
The centralizer of the diagonal subgroup $A$, $Z_H(A)$, has the following form:
$$Z_H(A)= \left\{\begin{bmatrix}B & \\ & C\end{bmatrix}: B \in \mathrm{GL}(m,\R) , C \in \mathrm{GL}(n,\R), \text{ and } \det B \det C =1\right\}.$$
\par The expanding horospherical subgroup of $A$, $U^{+}(A)$ has the following form:
$$U^{+}(A) := \left\{ u(X) := \begin{bmatrix}\I_m & X \\ & \I_n \end{bmatrix}: X \in \M(m\times n, \R)\right\}.$$
Similarly, the contracting horospherical subgroup $U^{-}(A)$ has the following form:
\begin{equation} \label{U-}
U^{-}(A) := \left\{u^{-}(X) := \begin{bmatrix}\I_m & \\ X & \I_n\end{bmatrix}: X \in \M(n\times m, \R)\right\}.
\end{equation}
For any $z \in Z_H(A)$ and $u(X) \in U^{+}(A)$, $z u(X) z^{-1} = u(zX)$ where $z \cdot X$ is defined as follows:
\begin{equation} \label{eq:ZonM}
\text{if } z = \begin{bmatrix}B & \\ & C\end{bmatrix}\in Z_{H}(A) \text{ and }
X\in \M(m\times n, \R) \text{ then } z\cdot X := BXC^{-1}. 
\end{equation}
 This defines an action of 
$Z_{H}(A)$ on $\M(m\times n, \R)$.

Similarly we can define the action of $Z_H(A)$ on $\M(n\times m, \R)$ induced by the conjugate action of $Z_H(A)$ on $U^{-}(A)$.
\par Let $P^{-}(A) := Z_H(A)U^{-}(A)$ denote the maximal parabolic subgroup of $H$ associated with $A$. 

\begin{definition} \label{SL2X}
 For any $X \in \mathrm{GL}(m,\R)$, we consider the following three elements in the Lie algebra $\mathfrak{h}$ of $H$: 
$$\mathfrak{n}^{+}(X) := \begin{bmatrix} \vect{0} & X & \vect{0} \\ \vect{0} & \vect{0} & \vect{0} \\ \vect{0} & \vect{0} & \vect{0}\end{bmatrix}, \qquad
\mathfrak{n}^{-}(X\inv) : = \begin{bmatrix} \vect{0} & \vect{0} & \vect{0} \\ X^{-1} & \vect{0} & \vect{0} \\ \vect{0} & \vect{0} & \vect{0}\end{bmatrix}, \qquad 
\mathfrak{a} := \begin{bmatrix}\I_m & \vect{0} & \vect{0} \\ \vect{0} & -\I_m & \vect{0} \\\vect{0} & \vect{0} & \vect{0}\end{bmatrix}.$$
Then $\{\mathfrak{n}^{+}(X), \mathfrak{n}^{-}(X\inv), \mathfrak{a}\}$ makes a {\em $\mathfrak{sl}(2,\R)$- triple\/}; that is, they satisfy the following relations
\[
[\mathfrak{a},\mathfrak{n}^{+}(X)]=2\mathfrak{n}^{+}(X) \quad 
[\mathfrak{a},\mathfrak{n}^{-}(X\inv)]=-2\mathfrak{n}^{-}(X\inv),\quad
[\mathfrak{n}^{+}(X), \mathfrak{n}^{-}(X\inv)]=\mathfrak{a}.
\]
Therefore, there is an embedding 
of $\mathrm{SL}(2,\R)$ into $H$ that sends $\begin{bmatrix}1 & 1 \\ 0 & 1\end{bmatrix}$ to $\exp(\mathfrak{n}^{+}(X))$, 
$\begin{bmatrix}1 & 0 \\ 1 & 1\end{bmatrix}$ to $\exp(\mathfrak{n}^{-}(X\inv))$, and $\begin{bmatrix}e^t & 0 \\ 0 & e^{-t}\end{bmatrix}$ 
to $\exp(t \mathfrak{a})$. We denote the image of this $\mathrm{SL}(2,\R)$ embedding by $\mathrm{SL}(2,X) \subset H$. Let us denote
$$\sigma(X) := \begin{bmatrix} \vect{0} & X & \vect{0} \\ -X^{-1} & 
\vect{0} &\vect{0} \\ \vect{0} & \vect{0} & \vect{0}\end{bmatrix} \in \mathrm{SL}(2,X).$$
It is easy to see that $\sigma(X)$ corresponds to $\begin{bmatrix}0 & 1 \\ -1 & 0\end{bmatrix} \in \mathrm{SL}(2,\R)$.
\end{definition}


\subsection{Unipotent invariance}

\par Recall that for $t >0$, $\mu_t$ denotes the normalized parametric measure on the curve $a(t) u(\varphi(I))x$, and $\mu_G$ denotes the unique $G$ invariant probability measure on $G/\Gamma$. Our aim is to prove that $\mu_t \rightarrow \mu_G$ as $t \rightarrow +\infty$. We first modify the measures $\mu_t$ to another measure $\lambda_t$ and show that if $\lambda_t \rightarrow \mu_G$,
then $\mu_t \rightarrow \mu_G$ as well. Then we could study $\{\lambda_t: t >0\}$ instead. The motivation for this modification is that any accumulation point of $\{\lambda_t : t >0 \}$ is invariant under a unipotent subgroup.


\par The measure $\lambda_t$ is defined as follows:
\begin{definition}[cf.~{\cite[(5.2)]{Shah2010}}]
 Without loss of generality, we may assume that $\varphi^{(1)}(s) \neq \vect{0}$ for all $s\in I$. Since $\varphi$ is analytic, there exists some integer $1 \leq b \leq m$, such that the derivative $\varphi^{(1)}(s)$ has rank $b$ for all $s \in I$ but finitely many points. Let $\E_b(m)$ be the $m$ by $m$ matrix defined as follows:
 $$\E_b(m) := \left\{ \begin{array}{cl}\begin{bmatrix} \I_b & \vect{0} \\ \vect{0} & \vect{0} \end{bmatrix} & \text{ if } b < m , \\ & \\ \I_m &  \text{ if } b = m.\end{array}\right.$$ 
 Given a closed subinterval $J \subset I$ such that $\varphi^{(1)}(s)$ has rank $b$ for all $s \in J$,
 we define an analytic curve $z: J \rightarrow Z_H(A)$ such that $z(s)\cdot\varphi^{(1)}(s) = [\E_b(m); \vect{0}]$ for all $s \in J$. For $t >0$, we define 
 $\lambda^{J}_t$ to be the normalized parametric measure on $\{z(s)a(t)u(\varphi(s))x: s \in J\}$; that is, for $f\in C_c(G/\Gamma)$,
 \begin{equation} \label{eq:lambda}
 \int f \dd \lambda^{J}_t : = \frac{1}{\abs{J}}\int_{s\in J} f(z(s)a(t)u(\varphi(s))x)\dd s.
 \end{equation}
\end{definition}


\begin{remark}
 For any subinterval $J\subset I$, we could similarly define $\mu^{J}_t$ to be the normalized parameter measure on 
  $a(t)u(\varphi(J))x$.
 \end{remark}
 \begin{proposition}
 Suppose that for any closed subinterval $J\subset I$ such that $\lambda^{J}_t$ is defined, we have $\lambda^{J}_t \rightarrow \mu_G$ as $t \rightarrow +\infty$.
 Then $\mu_t \rightarrow \mu_G$ as $t \rightarrow +\infty$.
\end{proposition}
\begin{proof}
\par Let $s_1, s_2, \dots , s_l \in I$ be all the points where $\varphi^{(1)}(s)$ does not have rank $b$. For any fixed $f \in C_c(G/\Gamma)$ and $\epsilon >0$, we want to show that for $t >0$ large enough,
$$\int f \dd \mu_t \overset{4\epsilon}{\approx} \int_{G/\Gamma} f \dd \mu_G .$$
For each $i \in \{1,2,\dots, l\}$, one could choose a small open subinterval $B_i \subset I$ containing $s_i$ such that 
\begin{equation}
\label{eq:Bi-1}
\left|(\sum_{i=1}^l | B _i | ) \int_{G/\Gamma} f \dd \mu_G\right| \leq \epsilon |I|,
\end{equation}
and for any $t >0$,
\begin{equation}
\label{eq:Bi-2}
\left| \int_{\cup_{i=1}^l B_i} f(a(t)u(\varphi(s))x) \dd s  \right| \leq \epsilon |I|.
\end{equation}
Since $f$ is uniformly continuous, there exists a constant $\delta>0$, such that if $x_1 \overset{\delta}{\approx} x_2$ then $f(x_1) \overset{\epsilon}{\approx} f(x_2)$. We cut $I\setminus \cup_{i=1}^l B_i$ into several small closed subintervals $J_1, J_2, \dots, J_p$ such that for every $J_r$, $z^{-1}(s_1)z(s_2) \overset{\delta}{\approx} e$ for any $s_1,s_2 \in J_r$. 
 \par Now for a fixed $J_r$, we choose $s_0 \in J_r$ and define $f_0(x)= f(z^{-1}(s_0)x)$. Then for any $s \in J_r$, because 
 $z^{-1}(s_0)z(s)a(t)u(\varphi(s))x \overset{\delta}{\approx} a(t)u(\varphi(s))x$, we have
 $$f_0(z(s)a(t)u(\varphi(s))x) = f(z^{-1}(s_0)z(s)a(t)u(\varphi(s))x) \overset{\epsilon}{\approx} f(a(t)u(\varphi(s))x).$$
 Therefore 
 $$\int f_0 \dd \lambda^{J_r}_t \overset{\epsilon}{\approx} \int f \dd \mu^{J_r}_t.$$
 Because $\int f_0 \dd \lambda^{J_r}_t \rightarrow \int_{G/\Gamma} f_0(x) \dd\mu_G(x)$ as $t \rightarrow +\infty$, and 
 $\int_{G/\Gamma} f_0(x) \dd\mu_G(x) = \int_{G/\Gamma}f(z^{-1}(s_0)x) \dd \mu_{G}(x) = \int_{G/\Gamma} f(x) \dd\mu_G$ (because $\mu_G$ is
 $G$-invariant), we have that there exists a constant $T_r>0$, such that for $t> T_r$, 
 $$\int f_0 \dd\lambda^{J_r}_t \overset{\epsilon}{\approx} \int_{G/\Gamma} f \dd\mu_G.$$
 Therefore, for $t > T_r$,
 $$\int f \dd\mu^{J_r}_t \overset{2\epsilon}{\approx} \int_{G/\Gamma} f \dd\mu_G,$$
 i.e.,
 $$\int_{J_r} f(a(t)u(\varphi(s))x) \dd s \overset{2\epsilon |J_r|}{\approx} |J_r| \int_{G/\Gamma} f \dd \mu_G. $$
 Then for $t > \max_{1\leq r\leq p} T_r$, we could sum up the above approximations for $r=1,2,\dots , p$ and get 
 $$\int_{I\setminus \cup_{i=1}^l B_i} f(a(t)u(\varphi(s))x) \dd s \overset{2\epsilon|I|}{\approx} (|I| -\sum_{i=1}^l |B_i|) \int_{G/\Gamma} f \dd \mu_G.$$
 Combined with (\ref{eq:Bi-1}) and (\ref{eq:Bi-2}), the above approximation implies that 
 $$\int_I f(a(t)u(\varphi(s))x) \dd s \overset{4\epsilon |I|}{\approx} |I| \int_{G/\Gamma} f \dd \mu_G,$$
 which is equivalent to 
$$\int f \dd \mu_t \overset{4\epsilon}{\approx} \int_{G/\Gamma} f \dd \mu_G . $$
 Because $\epsilon>0$ can be arbitrarily small, we complete the proof.
\end{proof}
\par By this proposition, if we could prove the equidistribution of $\{\lambda_t := \lambda_t^I : t >0\}$ as $t \rightarrow +\infty$ assuming that $\varphi^{(1)}(s)$ has rank $b$ for all $s \in I$, then the equidistribution of $\{\mu_t : t >0\}$ as $t \rightarrow +\infty$ will follow. Therefore, later in this paper, we will assume that $\varphi^{(1)}(s)$ has rank $b$ for all $s \in I$ and define $\lambda_t$ to be the normalised parametric measure on the curve $\{z(s)a(t)u(\varphi(s))x: s \in I\}$.
\par We will show that any limit measure of $\{\lambda_t : t >0\}$ is 
invariant under the unipotent subgroup 
\begin{equation} \label{eq:W}
W:= \{u(r[\E_b(m);\vect{0}]): r \in \R\}.
\end{equation}

\begin{proposition}[See \cite{Shah_1}]
 \label{prop_invariant_under_unipotent}
 Let $t_i \rightarrow +\infty$ be a sequence such that $\lambda_{t_i} \rightarrow \mu_{\infty}$ in weak-$\ast$ topology, then 
 $\mu_{\infty}$ is invariant under $W$-action.
\end{proposition}

\begin{proof}
Given any $f\in C_c(G/\Gamma)$, and $r \in \R$, we want to show that 
$$\int f(u(r [ \E_b(m) ; \vect{0}])x) \dd \mu_{\infty} = \int f(x) \dd \mu_{\infty}.$$
Since $z(s)$ and $\varphi(s)$ are analytic and defined on the closed interval $I=[a,b]$, there exists a constant $T_1 >0$ such that for $t \geq T_1$, $z(s)$ and $\varphi(s)$ can be extended to analytic curves defined on $[a - |r| e^{-(m+n)t} , b +|r| e^{-(m+n)t}]$. Throughout the proof, we always assume that $t_i \geq T_1$. Then $z(s + r e^{-(m+n)t_i})$ and $\varphi(s + r e^{-(m+n)t_i})$ are both well defined for all $s \in I$.

\par From the definition of $\mu_{\infty}$, we have 
$$
\int f(u(r[\E_b(m);\vect{0}])x)\dd\mu_{\infty}  =  \lim_{t_i \rightarrow +\infty} \frac{1}{\abs{I}}\int_{s \in I} f(u(r[\E_b(m);\vect{0}])z(s)a(t_i)u(\varphi(s))x) \dd s.
$$
We want to show that 
\begin{equation} \label{eq:invariant}
u(r[\E_b(m) ;\vect{0}])z(s)a(t_i)u(\varphi(s)) \approx z(s+r e^{-(m+n)t_i})a(t_i)u(\varphi(s+r e^{-(m+n)t_i})).
\end{equation}
Since $z(s+r e^{-(m+n)t_i}) \approx z(s)$ for $t_i$ large enough, it suffices to show that 
$$u(r[\E_b(m);\vect{0}])z(s)a(t_i)u(\varphi(s)) \approx z(s)a(t_i)u(\varphi(s+r e^{-(m+n)t_i})).$$
In fact,
$$\begin{array}{cl}
   & z(s)a(t_i)u(\varphi(s+r e^{-(m+n)t_i})) \\
   = & z(s) a(t_i)u(\varphi(s) + r e^{-(m+n)t_i} \varphi'(s) + O(e^{-2(m+n)t_i})) \\
   = & z(s) u(O(e^{-(m+n)t_i})) u(r \varphi'(s))a(t_i) u(\varphi(s)).
  \end{array}
$$
By the definition of $z(s)$, we have the above is equal to
$$ u(O(e^{-(m+n)t_i}))u(r [\E_b(m);\vect{0}])z(s)a(t_i)u(\varphi(s)).$$
When $t_i$ is large enough, $u(O(e^{-(m+n)t_i}))$ can be ignored. Therefore, for any $\delta >0$, there exists 
$T>0$, such that for $t_i > T$, 
$$u(r[\E_b(m);\vect{0}])z(s)a(t_i)u(\varphi(s)) \overset{\delta}{\approx} z(s+ r e^{-(m+n)t_i})a(t_i)u(\varphi(s+r e^{-(m+n)t_i})).$$
Now for any $\epsilon >0$, we choose $\delta>0$ such that whenever $x_1 \overset{\delta}{\approx} x_2$, we have 
$f(x_1) \overset{\epsilon}{\approx} f(x_2)$. Then from the above argument, we have for $t_i > T$,
$$f(u(r [\E_b(m);\vect{0}]) z(s) a(t_i)u(\varphi(s))x) \overset{\epsilon}{\approx} f(z(s+ r e^{-(m+n)t_i}) a(t_i) u(\varphi(s+ r e^{-(m+n)t_i}))x).$$
Therefore,
$$\begin{array}{cl} & \frac{1}{\abs{I}}\int_{s\in I} f(u(r [\E_b(m); \vect{0}]) z(s) a(t_i)u(\varphi(s))x)\dd s \\
                  \overset{\epsilon}{\approx} & \frac{1}{\abs{I}} \int_{s\in I} f(z(s+ r e^{-(m+n)t_i}) a(t_i) u(\varphi(s+ r e^{-(m+n)t_i}))x) \dd s \\
                  = & \frac{1}{\abs{I}} \int_{a+ r e^{-(m+n)t_i}}^{b+ r e^{-(m+n)t_i}} f(z(s)a(t_i)u(\varphi(s))x) \dd s.
\end{array}$$
It is easy to see that when $t_i>0$ is large enough,
$$\frac{1}{\abs{I}} \int_{a+ r e^{-(m+n)t_i}}^{b+ r e^{-(m+n)t_i}} f(z(s)a(t_i)u(\varphi(s))x) \dd s \overset{\epsilon}{\approx} \frac{1}{\abs{I}} \int_{a}^{b} f(z(s)a(t_i)u(\varphi(s))x) \dd s.$$
Therefore, for $t_i$ large enough,
$$\int f(u(r[ \E_b(m);\vect{0}])x) \dd\lambda_{t_i} \overset{2\epsilon}{\approx} \int f(x) \dd \lambda_{t_i}.$$
Letting $t_i \rightarrow +\infty$, we have 
$$\int f(u(r [\E_b(m);\vect{0}])x) \dd \mu_{\infty} \overset{2\epsilon}{\approx} \int f(x) \dd \mu_{\infty}.$$
Since the above approximation is true for arbitrary $\epsilon >0$, we have that 
$\mu_{\infty}$ is $W$-invariant.
\end{proof}


\subsection{Non-divergence of limit measures}
\par We also need to show that any limit measure $\mu_{\infty}$ of $\{\lambda_t: t>0\}$ is still a probability measure 
of $G/\Gamma$, i.e., no mass escapes to infinity as $t \rightarrow +\infty$. To do this, it suffices to show the following proposition:
\begin{proposition}
\label{prop_no_escape_mass}
 For any $\epsilon >0$, there exists a compact subset $\mathcal{K}_{\epsilon} \subset G/\Gamma$ such that 
 $\lambda_t(\mathcal{K}_{\epsilon}) \geq 1-\epsilon$ for all $t >0$.
\end{proposition}
\begin{remark}
In this proposition we only assume $\varphi$ is {\em generic}.
\end{remark}
\par This proposition will be proved via linearization technique combined with a lemma in linear dynamics as in \cite{Shah_2}. 

\begin{definition}
\label{def:representation_G}
 Let $\mathfrak{g}$ denote the Lie algebra of $G$, and denote $d = \dim G$. We define
 $$V = \bigoplus_{i=1}^{d} \bigwedge\nolimits^i \mathfrak{g},$$
 and let $G$ act on $V$ via $\bigoplus_{i=1}^d \bigwedge^i \mathrm{Ad}(G)$. This defines a linear representation of $G$:
 $$G \rightarrow \mathrm{GL}(V).$$
\end{definition}
\begin{remark}
In this paper, we will treat $V$ as a representation of $H$.
\end{remark}
The following theorem due to Kleinbock and Margulis is the basic tool to prove that there is no mass-escape when we
pass to a limit measure:
\begin{theorem}[see \cite{Dani} and \cite{Klein_Mar}]
 \label{non_divergence_theorem}
 Fix a norm $\norm{\cdot }$ on $V$. There exist finitely many vectors  $v_1, v_2, \dots , v_r \in V$ such that for each 
 $i=1,2,\dots, r$, the orbit $\Gamma v_i$ is discrete, and moreover, the
following holds: for any $\epsilon >0$ and $R > 0$, there exists a
compact set $K\subset G/\Gamma$ such that for any $t >0$ and any subinterval $J\subset I$, one of the
following holds:
\begin{enumerate}[label=\textbf{S.\arabic*}]
\item There exist $\gamma \in \Gamma$ and $j\in \{1,\dots , r\}$
such that $$\sup_{s\in J} \| a(t)u(\varphi(s)) g \gamma v_j \| < R,$$
\item $$|\{ s\in J:  a(t)u(\varphi(s))x \in K\}| \geq (1-\epsilon)|J|.$$
\end{enumerate}
\end{theorem}
\begin{remark}
 For the case $\varphi(s)$ is polynomial curve, the proof is due to Dani \cite{Dani}, for the case of analytic curve, 
 the proof is due to Kleinbock and Margulis \cite{Klein_Mar}. The crucial part of the proofs is to find some constants 
 $C>0$ and $\alpha>0$ such that in this particular representation, all the coordinate functions of $a(t)u(\varphi(\cdot))$ 
 are $(C, \alpha)$-good. Here a function $f: I \rightarrow \R$ is called
 $(C,\alpha)$-good if for any subinterval $J \subset I$ and any $\epsilon >0$, the following holds:
 $$|\{s\in J: |f(s)|<\epsilon\}| \leq C\left(\frac{\epsilon}{\sup_{s\in J}|f(s)|}\right)^{\alpha} |J|.$$
\end{remark}

\begin{notation}
\label{notation:weight-components}
 Let $\mathcal{V}$ be a finite dimensional linear representation of a Lie group $F$. Then for a one-parameter diagonal subgroup 
 $D=\{d(t): t \in \R\}$ of $F$, we could decompose $\mathcal{V}$ as the direct sum of eigenspaces of $D$; that is,
 $$\mathcal{V} = \bigoplus_{\lambda \in \R} \mathcal{V}^{\lambda}(D),$$
 where $\mathcal{V}^{\lambda}(D) = \{v\in \mathcal{V}: d(t)v = e^{\lambda t} v\}$.
 \par We define
 \[
 \mathcal{V}^{+}(D) = \bigoplus_{\lambda >0} \mathcal{V}^{\lambda}(D), \quad 
 \mathcal{V}^{-}(D) = \bigoplus_{\lambda <0 } \mathcal{V}^{\lambda}(D), \quad 
 \mathcal{V}^{\pm0}(D) = \mathcal{V}^{\pm}(D) + \mathcal{V}^0(D).
 \]
 For a vector $v \in \mathcal{V}$, we denote by $v^{+}(D)$ ($v^{\lambda}(D)$, $v^{-}(D)$, $v^0(D)$, $v^{+0}(D)$ and $v^{-0}(D)$ respectively) the projection of $v$ to $\mathcal{V}^{+}(D)$ ($\mathcal{V}^{\lambda}(D)$, $\mathcal{V}^{-}(D)$, $\mathcal{V}^0(D)$, $\mathcal{V}^{+0}(D)$ and 
 $\mathcal{V}^{-0}(D)$ respectively) with respect to the above direct sums.
 
\end{notation}

\par The proof of Proposition~\ref{prop_no_escape_mass} depends on the following property of finite dimensional representations of $\SL(m+n,\R)$:

\begin{lemma}[Basic Lemma]
\label{basic_lemma_1}
Let $V$ be a finite dimensional representation of $\SL(m+n, \R)$, and let $A =\{a(t): t\in \R\}\subset \SL(m+n,\R)$ denote the diagonal subgroup as in \eqref{eq:a(t)}. 
If an analytic curve 
$$\varphi: I=[a,b]\rightarrow \M(m\times n, \R)$$
is {\em generic\/}, then for any nonzero vector $v \in V$, there exists some $s \in I$ such that 
$$u(\varphi(s))v \not\in V^{-}(A).$$
\end{lemma}

A proof of this statement is one of the most important technical contributions of this paper, and we will postpone its proof to \S \ref{lemmas}.

\begin{proof}[Proof of Proposition \ref{prop_no_escape_mass} assuming Lemma \ref{basic_lemma_1}]
\par Let $V$ be as in Definition \ref{def:representation_G}. Since $A \subset H$ is a diagonal subgroup, we have the following decomposition: 
$$V = \bigoplus_{\lambda \in \R} V^{\lambda}(A)$$
where $V^{\lambda}(A)$ is defined as in Notation \ref{notation:weight-components}. Choose the norm $\norm{\cdot}$ on $V$ to be the maximum norm associated to some choices of norms on $V^{\lambda}(A)$'s.
 \par For contradiction we assume that there exists a constant $\epsilon>0$ such that for any compact subset $\mathcal{K} \subset G/\Gamma$, there exist some $t > 0$ such that $\lambda_{t}(\mathcal{K}) <1 - \epsilon$. Now we fix a sequence $\{R_i> 0: i \in \N \}$ tending to zero. By Theorem \ref{non_divergence_theorem}, for any $R_i$, there exists a 
 compact subset $\mathcal{K}_i\subset G/\Gamma$, such that for any $t >0$, one of the
following holds:
\begin{enumerate}[label=\textit{S\arabic*.}]
\item \label{divergence_condition} There exist $\gamma \in \Gamma$ and $j\in \{1,\dots , r\}$
such that $$\sup_{s\in I} \| a(t)u(\varphi(s)) g \gamma v_j \| < R_{i},$$
\item \label{non_divergence_condition} $$|\{ s\in I:  a(t)u(\varphi(s))x \in \mathcal{K}_i\}| \geq (1-\epsilon)\abs{I}.$$
\end{enumerate}
\par From our hypothesis, for each $\mathcal{K}_i$, there exists some $t_i>0$ such that \ref{non_divergence_condition} does not hold.
So there exist $\gamma_i \in \Gamma$ and $v_{j(i)}$ such that 
\begin{equation} \label{eq:vj}
\sup_{s\in I} \| a(t_i)u(\varphi(s)) g \gamma_i v_{j(i)} \| < R_i.
\end{equation}
By passing to a subsequence of $\{i\in \N\}$, we may assume that $v_{j(i)} = v_j$ remains the same for all $i$. 

Since $\Gamma v_j$ is discrete in $V$, we have $t_{i}\to \infty$ as $i\to\infty$ and there are the following two cases:
\begin{enumerate}[label=\textit{Case~\arabic*.}]
\item \label{case_1} By passing to a subsequence of $\{i\in \N\}$, $\gamma_i v_j= \gamma v_j$ remains the same for all $i$.
\item \label{case_2} $\|\gamma_i v_j\| \rightarrow \infty $ along some subsequence.
\end{enumerate}
\par For \ref{case_1}: We have $a(t_i) u(\varphi(s)) g \gamma v_j \rightarrow \vect{0}$ as $i \rightarrow \infty$ for all $s \in I$. This implies that 
$$\{u(\varphi(s)) g \gamma v_j\}_{s\in I} \subset V^{-}(A),$$
which contradicts Lemma \ref{basic_lemma_1}.
\par For \ref{case_2}: After passing to a subsequence, we have
 \begin{equation} \label{eq:v}
 v := \lim_{i\rightarrow \infty} g \gamma_i v_j/{\norm{g \gamma_i v_j}}, \quad
\norm{v}=1\text{, and } \lim_{i\to\infty}\norm{g \gamma_i v_j} = \infty.
\end{equation}
By Lemma~\ref{basic_lemma_1}, let $s\in I$ be such that $u(\varphi(s))v\not\in V^{-}(A)$. Then by \eqref{eq:v} there exists $\delta_{0}>0$ and $i_{0}\in\N$ such that 
\[
\norm{(u(\varphi(s))g \gamma_i v_j)^{0+}}\geq \delta_{0} \norm{g \gamma_i v_j}, \quad \forall i\geq i_{0}.
\]
Then 
\[
\norm{a(t_{i})u(\varphi(s))g \gamma_i v_j}\geq \delta_{0} \norm{g \gamma_i v_j}\to\infty,
\text{ as $i\to\infty$},
\] 
which contradicts \eqref{eq:vj}. Thus Cases~1~and~2  both lead to contradictions. 
\end{proof}
\begin{remark}
\label{rmk:non-divergence}
 The same proof also shows that any limit measure of $\{\mu_t : t >0\}$ 
is still a probability measure, which is the non-divergence part of Theorem \ref{goal_thm}. 
\end{remark}


\section{Ratner's theorem and linearization technique}
\label{ratner_linearization}

Take any convergent subsequence $\lambda_{t_i} \rightarrow \mu_{\infty}$. By Proposition~\ref{prop_invariant_under_unipotent} and Proposition~
\ref{prop_no_escape_mass}, $\mu_{\infty}$ is a $W$-invariant probability measure on $G/\Gamma$, where $W$ is a unipotent one-parameter subgroup given by \eqref{eq:W}. We will apply Ratner's theorem and the linearization technique to understand the measure $\mu_{\infty}$. 

\begin{notation}
\label{def:for_linearization}
 Let $\mathcal{L}$ be the collection of proper analytic subgroups $L < G$ such that $L\cap \Gamma$ is a lattice of $L$. Then $\mathcal{L}$ is a countable set (\cite{Ratner}).
\par For $L\in \mathcal{L}$, define:
$$N(L,W):= \{g\in G: g^{-1}Wg\subset L\}  
\text{, and }
S(L,W):= \bigcup_{L'\in \mathcal{L}, L' \subsetneq L} N(L', W).$$
\end{notation}
\par We formulate Ratner's measure classification theorem as follows (cf. \cite{MS1995}):
\begin{theorem}[\cite{Ratner}]
\label{ratner} Given the $W$-invariant probability measure $\mu$ on
$G/\Gamma$, if $\mu$ is not $G$-invariant then there exists $L \in \mathcal{L}$ such that
\begin{equation}
\begin{array}{ccc}
\mu(\pi(N(L,W)))>0 & \text{ and } & \mu(\pi(S(L,W)))=0.
\end{array}
\end{equation}
Moreover, almost every $W$-ergodic component of $\mu$ on
$\pi(N(L,W))$ is a measure of the form $g\mu_L$ where $g\in
N(L,W)\backslash S(L,W)$, $\mu_L$ is a finite $L$-invariant measure
on $\pi(L)$, and $g\mu_L(E)=\mu_L (g^{-1}E)$ for all Borel sets
$E\subset G/\Gamma$. In particular, if $L\lhd  G$, then the restriction of $\mu$ on $\pi(N(L,W))$ is
$L$-invariant.
\end{theorem}
\par If $\mu_{\infty}=\mu_G$, then there is nothing to prove. So we may assume $\mu_{\infty}\neq \mu_G$. Then by Ratner's Theorem, 
there exists $L \in \mathcal{L}$ such that 
\begin{equation} \label{eq:L}
\mu_{\infty}(\pi(N(L,W))) > 0 \text{ and }  \mu_{\infty}(\pi(S(L,W))) =0.
\end{equation}

Now we want to apply the linearization technique to obtain algebraic consequences of this statement.

\begin{notation}
Let $V$ be the finite dimensional representation of $G$ defined as in Definition \ref{def:representation_G}, for $L \in \mathcal{L}$,
we choose a basis $\mathfrak{e}_1, \mathfrak{e}_2, \dots, \mathfrak{e}_l$ of the Lie algebra $\mathfrak{l}$ of $L$, and define 
$$p_L = \wedge_{i=1}^l \mathfrak{e}_i \in V.$$
Define 
$$\Gamma_L := \left\{\gamma\in \Gamma: \gamma p_L = \pm p_L \right\}.$$
From the action of $G$ on $p_L$, we get a map:
$$\begin{array}{l} \eta:  G \rightarrow V,  \\
                     g \mapsto g p_L .
   \end{array}
$$
Let $\mathcal{A}$ denote the Zariski closure of $\eta(N(L,W))$ in $V$. Then $N(L,W)=G\cap \eta^{-1}({\mathcal A})$. 
\end{notation}

Using the fact that $\varphi$ is analytic, we obtain the following consequence of the linearization technique (cf.\ \cite{Shah_1,Shah_2,Shah2010}).

\begin{proposition}[{\cite[Proposition~5.5]{Shah_2}}]
\label{relative_small prop} Let $C$ be a compact subset of $N(H,W)\setminus S(H,W)$. Given $\epsilon > 0$, there exists a
compact set $\mathcal{D}\subset \mathcal{A}$ such that, 
given a relatively compact neighborhood $\Phi$ of
$\mathcal{D}$ in $V$, there exists a neighborhood $\mathcal{O}$
of $C\Gamma$ in $G/\Gamma$ such that for any $t\in\R$ and subinterval
$J\subset I$, one of the following statements holds:
\begin{enumerate}[label=\textit{SS\arabic*.}]
\item \label{1} $|\{s\in J: a(t)u(\varphi(s))g\Gamma \in \mathcal{O}\}|\leq \epsilon \abs{J}$.
\item \label{2} There exists $\gamma\in\Gamma$ such that $a(t)z(s)u(\varphi(s))g\gamma p_{L} \in \Phi$ for all $s\in J$.
\end{enumerate}
\end{proposition}

The following proposition provides the obstruction to the limiting measure not being $G$-invariant in terms of linear actions of groups, and it is a key result for further investigations. 

\begin{proposition}
 \label{prop_algebraic_condition}
 There exists a $\gamma \in \Gamma$ such that
\begin{equation}\label{eq:V0-}
\{u(\varphi(s))g\gamma p_L\}_{s\in I} \subset V^{-0}(A).
\end{equation}
\end{proposition}

\begin{proof}[Proof (assuming Lemma~\ref{basic_lemma_1}).]
By \eqref{eq:L}, there exists a compact subset $C \subset N(L,W))\setminus S(L,W)$  and $\epsilon>0$ such that $\mu_{\infty}(C\Gamma) >\epsilon>0$. Apply Proposition~\ref{relative_small prop} to obtain $\mathcal{D}$, and choose any $\Phi$, and obtain a $\mathcal{O}$ so that either \ref{1} or \ref{2} holds.  Since $\lambda_{t_{i}}\not\to \mu_{\infty}$, we conclude that \ref{1} does not hold for $t=t_{i}$ for all $i\geq i_{0}$. Therefore  for every $i\geq i_{0}$, \ref{2} holds and there exists $\gamma_{i}\in\Gamma$ such that
\begin{equation} \label{eq:Phi}
a(t_{i})z(s)u(\varphi(s))g\gamma_{i}p_{L}\subset \Phi.
\end{equation}

Since $\Gamma p_{L}$ is discrete in $V$, by passing to a subsequence, there are two cases:
\begin{enumerate}
\item[{\it Case 1.}] $\gamma_{i}p_{L}=\gamma p_{L}$ for some $\gamma\in\Gamma$ for all $i$; or 
\item[{\it Case 2.}] $\norm{\gamma_{i}p_{L}}\to\infty$ as $i\to \infty$. 
\end{enumerate}

In {\it Case~1}, since $\Phi$ is bounded in \eqref{eq:Phi}, we deduce that $z(s)u(\varphi(s))g\gamma p_{L}\subset V^{-0}(A)$ for all $s\in I$. Since $V^{-0}(A)$ is $Z_{H}(A)$-invartiant, \eqref{eq:V0-} holds. 

In {\it Case~2}, by arguing as in the \ref{case_2} of the Proof of Proposition~\ref{prop_no_escape_mass}, using genericness of $\varphi$ and Lemma~\ref{basic_lemma_1}, we obtain that $\norm{a(t_{i})u(\varphi(s))g\gamma_{i}p_{L}}\to\infty$. This contradicts \eqref{eq:Phi}, because $z(s)\subset Z_H(A)$ and $\Phi$ is bounded. Thus {\it Case~2} does not occur.
\end{proof}

\par Our goal is to obtain an explicit geometric condition on $\varphi(I)$ which implies that the linear algebraic condition \eqref{eq:V0-} does not hold.


\section{Some linear dynamical results}
\label{lemmas}

We shall start with a dynamical lemma about finite dimensional representations of $\SL(2,\R)$ which sharpens the earlier results due to Shah \cite[Lemma 2.3]{Shah_1} and Yang \cite{Yang_1}. 

\begin{lemma}
\label{yang_lemma}

\par Let $V$ be a finite dimensional linear representation of $\mathrm{SL}(2,\R)$. Let
 \[
 A = \left\{a(t):= \begin{bmatrix}e^t&\\&e^{-t}\end{bmatrix}: t \in \R\right\},\
 U = \left\{u(s) := \begin{bmatrix}1 & s \\ 0 & 1\end{bmatrix}: s \in \R\right\},\
 U^{-} =\left\{u^{-}(s) := \begin{bmatrix}1 & 0 \\ s & 1\end{bmatrix}: s \in \R\right\}.
 \]
 Express $V$ as the direct sum of eigenspaces with respect to the action of $A$:
 \[
 V = \bigoplus_{\lambda\in\R } V^{\lambda}(A) 
 \text{, where } 
 V^{\lambda}(A) := \{v \in V: a(t)v = e^{\lambda t} v : \forall t \in \R\}.
 \]
 For any $v \in V\setminus\{0\}$ and $\lambda\in\R$, let $v^{\lambda}=v^{\lambda}(A)$ denote the $V^{\lambda}(A)$-component of $v$,  $\lambda^{\max}(v)=\max\{\lambda:v^{\lambda}\neq 0\}$, and $v^{\max}=v^{\lambda^{\max}}(v)$. Then for any $r\neq 0$, 
 \begin{equation} \label{eq:max-urv}
 \lambda^{\max}(u(r)v) \geq - \lambda^{\max}(v).
 \end{equation}
 
 In particular, 
\begin{equation} \label{eq:-+}
\text{if $\lambda^{\max}(v)<0$ then $\lambda^{\max}(u(r)v)>0$, $\forall r\neq 0$.}
\end{equation}

Moreover, if the equality holds in \eqref{eq:max-urv} then 
 \begin{equation} \label{eq:equality}
 v = u^{-}(-r^{-1}) v^{\max}\text{ and } (u(r)v)^{\max} = \sigma(r) v^{\max} \text{, where } 
 \sigma(r) = \begin{bmatrix}0 & r \\ -r^{-1} & 0\end{bmatrix}. 
 \end{equation}
\end{lemma}

\begin{proof}
Observe that $u(1)u^{-}(-1)u(1)=\sigma(1)$, $u(-1)u^{-}(1)u(-1)=\sigma(-1)$ and for $r\neq 0$, conjugating all terms of these equalities by $a(\log(\abs{r}/2))$ we get $u(r)u^{-}(-r^{-1})u(r)=\sigma(r)$, and hence
\begin{equation} \label{eq:ur}
u(r)=\sigma(r)u(-r)u^{-}(r^{-1}), \, \forall r\neq 0.
\end{equation}

Since $\sigma(r) a(t) \sigma(r)^{-1}=a(-t)$ for all $r\neq 0$, we have that
\[
\sigma(r)V^{\lambda}(A)=V^{-\lambda}(A)\text{, for all $\lambda$.}
\]
Hence for any $v\in V\setminus\{0\}$, 
\begin{equation} \label{eq:sigma}
\lambda^{\max}(\sigma(r)v)=-\lambda^{\min}(v) \text{, and }
(\sigma(r)v)^{\max}=\sigma(r)v^{\min}. 
\end{equation}
For any $r\in\R$, since $u(r)$ is unipotent and $a(t)u(r)a(-t)=u(e^{2t}r)$, we have that
\begin{equation} \label{eq:min}
\lambda^{\min}(u(r)v)=\lambda^{\min}(v).
\end{equation}
Similarly, for any $s\in\R$, we have $a(t)u^{-}(s)a(-t)=u^{-}(e^{-2t}s)$, and hence
\begin{equation} \label{eq:max}
\lambda^{\max}(u^{-}(s)v)=\lambda^{\max}(v).
\end{equation}
Using the above relations \eqref{eq:ur}, \eqref{eq:sigma}, \eqref{eq:min} and \eqref{eq:max}, we get
\begin{align*}
\lambda^{\max}(u(r)v)
=\lambda^{\max}(\sigma(r)u(-r)u^{-}(r^{-1})v)
&=-\lambda^{\min}(u(-r)u^{-}(r^{-1})v)\\
&=-\lambda^{\min}(u^{-}(r^{-1})v)
\geq -\lambda^{\max}(u^{-}(r^{-1})v)
= -\lambda^{\max}(v).
\end{align*}

Further if there are all equalities in the above relation, then
\[
\lambda^{\min}(u^{-}(r^{-1})v)=\lambda^{\max}(u^{-}(r^{-1})v)=\lambda^{\max}(v).
\]
Therefore,
\[
u^{-}(r^{-1})v= (u^{-}(r^{-1})v)^{\max}=v^{\max}\text{; that is, } 
v = u^{-}(-r^{-1})v^{\max},
\]
and
\[
(u(r)v)^{\max}=\sigma(r)(u(-r)u^{-}(r^{-1})v)^{\min}=\sigma(r)(u^{-}(r^{-1})v)^{\min}=\sigma(r)(u^{-}(r^{-1})v)^{\max}=\sigma(r)v^{\max}.
\]
\end{proof}

\par Lemma \ref{yang_lemma} immediately implies the following statement:

\begin{corollary} \label{cor:sl2}

Let the notation be as in Lemma~\ref{yang_lemma}. If $v,u(r)v\in V^{0-}(A)$ for some $r\neq 0$, then $\lambda^{\max}(v)=0$ and $v=  u^{-}(r^{-1})v^{0}$, where $v^{0}\neq \vect{0}$ denotes the $V^{0}(A)$ component of $v$.
\end{corollary}

\subsection{Linear dynamical lemmas for $\SL(m+n,\R)$ representations}

First we give the proof of the basic lemma (Lemma \ref{basic_lemma_1}) that we used more than once in previous sections. The new techniques developed in this section forms the core of this paper, and we expect these techniques to be valuable for other problems. 


\begin{proof}[Proof of Lemma \ref{basic_lemma_1}]
We use induction to complete the proof. For the case $m=n$, the lemma is due to Yang \cite{Yang_1}. We provide a proof here for the sake of self-containedness. 

When $m=n$, we take a point $s_0$ and a subinterval $J_{s_0} \subset I$ such that for all $s \in J_{s_0}$, $\varphi(s)-\varphi(s_0) \in \mathrm{GL}(m,\R)$. Then we consider the subgroup $\SL(2,\varphi(s)- \varphi(s_0)) \cong \SL(2,\R) \subset \SL(2m,\R)$ for some fixed $s \in J_{s_0}$ (see Definition \ref{SL2X}), and apply Corollary~\ref{cor:sl2} for $\SL(2,\R)$ replaced by $\SL(2, \varphi(s)- \varphi(s_0))$, $v$ replaced by $u(\varphi(s_0))v$ and 
$u(r)$ replaced by $u(\varphi(s)-\varphi(s_0))$. Then one of $u(\varphi(s_0))v$ and $u(\varphi(s))v$ is not contained in $V^{-}(A)$. This proves the statement for $m=n$.

 If $m>n$ then by applying a suitable inner automorphism of $\SL(m+n,\R)$ given by a coordinate permutation $\sigma_{m,n}$, we can covert this problem to the case of $m<n$. Therefore we will assume that $m<n$. 

 As inductive hypothesis, we assume that for all $(m', n')$ such that $m'\leq m$, $n'\leq n$ and $m'+n' <m+n$, the result holds. We want to prove the result holds for $(m,n)$. 
  
For contradiction, we assume that for some nonzero vector $v \in V$, such that 
$$u(\varphi(s))v \in V^{-}(A)$$ for all 
$s \in I$. For $s\in I$, let $\mu_{0}(s)=\max\{\lambda:(u(\varphi(s))v)^{\lambda}(A)\neq 0\}$ and 
$\mu_{0}=\max\{\mu_0(s):s\in I\}$. Since $\varphi$ is analytic, we have $\mu_{0}(s)=\mu_{0}$ for all but finitely many $s\in I$. Also by our assumption we have that
\begin{equation} \label{eq:mu0}
\mu_{0}<0.
\end{equation}

Fix $s_0 \in I$ and a subinterval $J_{s_0} \subset I$ such that $\mu_{0}(s)=\mu_{0}$ for all $s\in J_{s_{0}}$ and if we write $\varphi(s) = [\varphi_1(s); \varphi_2(s)]$, then $\varphi_1(s)-\varphi_1(s_0) \in \GL(m,\R)$ for $s \in J_{s_0}$. Let
$$\psi: J_{s_0} \rightarrow \M(m\times (n-m), \R) \text{, be defined by } \psi(s):= (\varphi_1(s)-\varphi_1(s_0))^{-1}(\varphi_2(s)-\varphi_2(s_0)).$$
Then $\psi$ is {\em generic} by the of genericness of $\varphi$ (see Definition \ref{def:geometric_condition}). Replacing $v$ by $u(\varphi(s_0))v$ and $\varphi(s)$ by $\varphi(s)-\varphi(s_0)$, we may assume that $\varphi(s_0) = \vect{0}$.

For any fixed $s \in J_{s_0}$, it is straightforward to verify that 
\begin{gather} \label{eq:uprime1}
u(\varphi(s)) = u'(-\psi(s))u([\varphi_1(s); \vect{0}]) u'(\psi(s)) \text{,  where } \\
u'(Y):= \begin{bmatrix}\I_m & & \\ & \I_m & Y \\ & & \I_{n-m}\end{bmatrix} \in Z_{H}(A) \text{ for $Y\in M(m\times (n-m),\R)$}. \label{eq:uprime2}
\end{gather}
Therefore $u(\varphi(s))v \in V^{-}(A)$ implies that 
$$u([\varphi_1(s); \vect{0}]) u'(\psi(s))v \in V^{-}(A).$$
Let us denote 
$$A_1 :=\left\{a_1(t):= \begin{bmatrix}e^{t} \I_m & & \\ & e^{-t}\I_m & \\ & & \I_{n-m}\end{bmatrix}: t \in \R \right\},$$
and 
$$A_2 :=\left\{a_2(t):= \begin{bmatrix} \I_m & & \\ & e^{(n-m)t}\I_m & \\ & & e^{-mt}\I_{n-m}\end{bmatrix}: t \in \R \right\}.$$
We express $V$ as the direct sum of common eigenspaces of $A_1$ and $A_2$:
\begin{equation} \label{eq:Vdelta12}
V = \bigoplus_{\delta_1, \delta_2} V^{\delta_1, \delta_2} \text{, where } 
V^{\delta_1, \delta_2}:= \left\{v \in V: a_1(t)v = e^{\delta_1 t}v, a_2(t)v = e^{\delta_2 t}v \text{ for all } t \in \R\right\}.
\end{equation}
Then because $a(t) = a_1(nt)a_2(t)$, we have 
\begin{equation} \label{eq:V-lambda}
V^{\lambda}(A) =\bigoplus_{n\delta_1+\delta_2 = \lambda} V^{\delta_1,\delta_2}.
\end{equation}
For any vector $v \in V$, let $v^{\delta_1, \delta_2}$ denote the projection of $v$ onto the eigenspace $V^{\delta_1,\delta_2}$. 

We also decompose $V$ as the direct sum of irreducible sub-representations of 
$A \ltimes\SL(2,\varphi_1(s))$. For any such sub-representation $W\subset V$, let $p_{W}:V\to W$ denote the $A$-equivariant projection.  By the theory of finite dimensional representations of $\SL(2,\R)$, there exists a basis $\{w_0, w_1, \dots, w_r\}$ of $W$ such that 
\begin{equation} \label{eq:a1}
a_1(t)w_i = e^{(r-2i)t}w_i,\quad \text{for }0\leq i\leq r.
\end{equation}
We claim that each $w_i$ is also an eigenvector for $A$. In fact, 
$$a(t)= a_1((m+n)t/2) b(t), \quad \text{where } b(t)= \begin{bmatrix}e^{\frac{n-m}{2}t}\I_{2m} & \\ & e^{-mt}\I_{n-m}\end{bmatrix}\in Z_{H}(\SL(2,\varphi_1(s)),
$$
and hence $b(t)$ acts on $W$ as a scaler $e^{\delta t}$ for some $\delta\in\R$. Therefore, 
\begin{equation} \label{eq:awi}
a(t)w_i = e^{((r-2i)(m+n)/2+\delta)t}w_i, \quad\text{for } 1\leq i\leq r.
\end{equation}
Since $(m+n)/2>0$, if $k<i$ then the $A$-weight of $w_{k}$ is strictly greater than the $A$-weight of $w_{i}$. 
\par Since $u'(\psi(s))\in Z_{H}(A)$, $\mu_{0}$ is the highest $A$-weight for $v$, we have that $\mu_{0}$ is also the highest $A$-weight for $u'(\psi(s))v$ and 
\[
(u'(\psi(s))v)^{\mu_0}(A) = u'(\psi(s))v^{\mu_0}(A).
\]

Now suppose that $W$ as above is such that $p_W(u'(\psi(s))v^{\mu_0}(A))\neq \vect{0}$.  Then 
\[
p_W(u'(\psi(s))v^{\mu_0}(A)) = a_i w_i, \text{ for some } 0\leq i\leq r, \ 0\neq a_{i}\in\R;
\]
by \eqref{eq:awi} $\mu_{0}=(r-2i)(m+n)/2+\delta$. For $k<i$, the weight of $w_k$ for $A_1$ is greater than that of $w_i$, so the $A$-weight of $w_k$ is greater than the $A$-weight of $w_{i}$ which equals $\mu_0$. Since the projection $p_{W}$ is $A$-equivariant and $\mu_{0}$ is the highest $A$-weight, we have
\[
p_W(u'(\psi(s))v) = \sum_{k\geq i} a_k w_k, \quad \text{where } a_{k}\in\R.
\]

We claim that $r-2i \geq 0$. In fact, if $r-2i<0$, then by \eqref{eq:a1}, $p_W(u'(\psi(s))v) \in V^{-}(A_1)$. By Corollary~\ref{cor:sl2},
$$V^{-0}(A_{1})\not\ni u([\varphi_1(s); \vect{0}])p_W(u'(\psi(s))v) = p_W(u([\varphi_1(s); \vect{0}]) u'(\psi(s))v). $$
So $p_W(u([\varphi_1(s); \vect{0}]) u'(\psi(s))v)$ must have nonzero projection on $\R\/w_{k}$ for some $k<i$. Hence $u([\varphi_1(s); \vect{0}]) u'(\psi(s))v$ has nonzero projection $V^{\mu}(A)$ for some $\mu > \mu_0$. Now since $u'(-\psi(s))\in Z_{H}(A)$, the projection of $u(\varphi(s))v =u'(-\psi(s))u([\varphi_1(s);\vect{0}])u'(\psi(s))v$ on $V^{\mu}(A)$ is nonzero for $\mu>\mu_0$. This contradicts our choice of $\mu_{0}$ and proves the claim that $r-2i\geq 0$.

\par This claim implies that for any $(\delta_1, \delta_2)$, if $(u'(\psi(s))v^{\mu_0}(A))^{\delta_1,\delta_2}\neq \vect{0}$ then $\delta_1\geq 0$. Since
$\mu_0 =n\delta_1 + \delta_2 <0$, and we have 
$\delta_2 < 0$. In other words, 
$$u'(\psi(s))v^{\mu_0}(A) \in V^{-}(A_2),\ \forall\/ s\in J_{s_0}.$$
\par Now $u'(\psi(s))$ and $A_2$ are both contained in 
$$\begin{bmatrix}\I_m & \\ & \SL(n,\R)\end{bmatrix} \cong \SL(m+(n-m), \R).$$
Our inductive hypothesis for $(m,n-m)$ tells that this is impossible because $\psi$ is {\em generic\/}. 
\par This finishes the proof.
\end{proof}


\par For {\em supergeneric\/} curves, we want to obtain the following stronger conclusion. 
\begin{lemma}
\label{basic_lemma_2}
Let $V$ be an irreducible representation of 
$H=\SL(m+n,\R)$. Let 
$$\varphi: I=[a,b] \rightarrow \M(m\times n, \R)$$
be a {\em supergeneric\/} analytic curve. Then if there is a nonzero vector $v \in V$ such that
$$u(\varphi(s))v \in V^{-0}(A),$$
for all $s \in I$, then $V$ is a trivial representation.
\end{lemma}

\par To prove this lemma, we will need the following observation. 
\begin{lemma}
\label{derivative_lemma}
Let $V$ be a finite dimensional representation of $\SL(m+n,\R)$ and let 
$$A := \left\{a(t):= \begin{bmatrix}e^{nt}\I_m & \\ & e^{-mt}\I_n\end{bmatrix}: t\in \R\right\}.$$
Let
$$\varphi: I=[a,b]\rightarrow \M(m\times n, \R)$$
be an analytic curve. Suppose there exists a nonzero vector 
$v \in V$ such that 
$$u(\varphi(s))v \in V^{-0}(A),$$
for all $s \in I$. Then for all $s \in I$, $(u(\varphi(s))v)^0(A)$
is invariant under the unipotent flow $\{u(h \varphi^{(1)}(s)): h \in \R\}$.
\end{lemma}
\begin{proof}[Proof of Lemma \ref{derivative_lemma}]
For any $h \in \R$, on the one hand, 
$$a(t)u(\varphi(s + e^{-(m+n)t} h)) v = (u(\varphi(s + e^{-(m+n)t} h)) v)^0(A) + O(e^{-\lambda(m,n)t}),$$
for some $\lambda(m,n)>0$ depending on $m$ and $n$. As $t \rightarrow \infty$, 
\[
(u(\varphi(s + e^{-(m+n)t} h)) v)^0(A) \rightarrow (u(\varphi(s)) v)^0(A) \text {, and }
O(e^{-\lambda(m,n)t}) \rightarrow \vect{0}.
\]
Thus, as $t \rightarrow \infty$,
$$a(t)u(\varphi(s + e^{-(m+n)t} h)) v \rightarrow (u(\varphi(s))v)^{0}(A) .$$
On the other hand, 
$$\begin{array}{cl} & a(t)u(\varphi(s + e^{-(m+n)t} h)) v \\
= & a(t) u(h e^{-(m+n)t}\varphi^{(1)}(s)) u(O(e^{-2(m+n)t})) u(\varphi(s)) v \\
= & a(t) u(h e^{-(m+n)t}\varphi^{(1)}(s))a(-t) a(t)u(O(e^{-2(m+n)t}))a(-t) a(t)u(\varphi(s)) v \\
= & u(h \varphi^{(1)}(s)) u(O(e^{-(m+n)t})) a(t)u(\varphi(s)) v .\end{array} $$
As $t \rightarrow \infty$, $u(O(e^{-(m+n)t})) \rightarrow \mathrm{id}$, $a(t)u(\varphi(s)) v \rightarrow  (u(\varphi(s))v)^0(A)$. Therefore, as $t \rightarrow \infty$,
$$a(t)u(\varphi(s + e^{-(m+n)t} h)) v \rightarrow  u(h \varphi^{(1)}(s))(u(\varphi(s))v)^0(A).$$
This shows that $(u(\varphi(s))v)^{0}(A)$ is invariant under $\{u(h \varphi^{(1)}(s) ): h\in \R\}$.
\end{proof}

\begin{proof}[Proof of Lemma~\ref{basic_lemma_2}] The strategy of the proof is similar to that of Lemma~\ref{basic_lemma_1}. 
\par We begin with the case $m=n$. This case is studied in \cite{Yang_1} but the statement proved there is weaker than the statement here.

Fix a point $s_0 \in I$ and a subinterval $J_{s_0} \subset I$ such that 
$\varphi(s)-\varphi(s_0)$ is invertible for all 
$s\in J_{s_0}$ and moreover, $\{ \mathfrak{n}^{-}((\varphi(s_1) -\varphi(s_0))^{-1} - (\varphi(s_2) -\varphi(s_0))^{-1}): s_1, s_2 \in J_{s_0}\}$ is not contained in any proper observable subalgebra of $\mathfrak{sl}(2m,\R)$. By replacing $\varphi(s)$ by 
$\varphi(s)-\varphi(s_0)$, we may assume that 
$\varphi(s_0) = \vect{0}$.

\par  In the isomorphism $\SL(2,\R) \cong \SL(2, \varphi(s))$ (see Definition \ref{SL2X}), $\begin{bmatrix}1 & 1 \\ 0 & 1\end{bmatrix}$ corresponds to $u(\varphi(s))$, $\begin{bmatrix}1 & 0 \\ 1 & 1\end{bmatrix}$ corresponds to $u^{-}(\varphi^{-1}(s))$, and $\begin{bmatrix}0 & 1 \\ -1 & 0\end{bmatrix}$ corresponds to $\sigma(\varphi(s))$. By Corollary \ref{cor:sl2}, we have that 
$v, u(\varphi(s))v \in V^{-0}(A)$ implies that 
$$v = u^{-}(\varphi^{-1}(s)) v^0(A).$$
In particular, $v^0(A)\neq \vect{0}$. 

Taking any $s_1, s_2 \in J_{s_0}$, we have 
$$u^{-}(\varphi^{-1}(s_1)) v^0(A) =v = u^{-}(\varphi^{-1}(s_2)) v^0(A).$$
This shows that $v^0(A)$ is fixed by 
$u^{-}(\varphi^{-1}(s_1)- \varphi^{-1}(s_2))$ for all $s_1, s_2 \in J_{s_0}$. By definition, $v^0(A)$ is also fixed by $A$. Let $L$ denote the subgroup of $H$ stabilizing $v^{0}(A)$, and $\mathfrak{l}$ denote its Lie algebra. Then from the above argument we have $\mathfrak{l}$ is observable and contains $\mathcal{E}\in \mathrm{Lie}(A)$ (see \eqref{eq:E}) and 
$$\{  \mathfrak{n}^{-}((\varphi(s_1)-\varphi(s_0))^{-1} -(\varphi(s_2)-\varphi(s_0))^{-1}): s_1, s_2 \in J_{s_0}\};$$
recall that earlier we had replaced $\varphi(s)$ by $\varphi(s)-\varphi(s_{0})$ and assumed that $\varphi(s_{0})=\vect{0}$ for notational simplicity. Because $\varphi$ is {\em supergeneric\/}, in view of \eqref{eq:supergeneric} we have that $L = H$. Since $V$ is an irreducible representation of $H$, $V$ is trivial. 
\par This finishes the proof for $m=n$.

For the general case we give the proof by an inductive argument.
Suppose the statement holds for all $(m',n')$ such that $m' \leq m$, $n' \leq n$ and $m'+n' < m+n$. We want to prove the statement for $(m,n)$.

\par We choose a point $s_0$ and a subinterval $J_{s_0} \subset I$ such that the following statements hold:
\begin{enumerate}
\item If we write $\varphi(s) = [\varphi_1(s); \varphi_2(s)]$ where $\varphi_1(s)$ is the first $m$ by $m$ block, and $\varphi_2(s)$ is the rest $m$ by $n-m$ block, then for any $s \in J_{s_0}$, $\varphi_1(s) - \varphi_1(s_0)$ is invertible.
\item The curve $\psi(s) = (\varphi_1(s)-\varphi_1(s_0))^{-1}(\varphi_2(s)-\varphi_2(s_0))$ is {\em supergeneric\/} as a curve from $J_{s_0}$ to $\M(m\times(n-m), \R)$.
\end{enumerate}

Without loss of generality we may assume that $\varphi(s_0) = \vect{0}$ and 
$v \in V^{-0}(A)$.
The notations such that $u'(\cdot)$, $A_2$ and $v^{\mu_0}(A)$ have the same meaning as in the proof of Lemma \ref{basic_lemma_1}. Using the same argument as the proof of Lemma \ref{basic_lemma_1}, we could deduce that 
$$u'(\psi(s)) v^{\mu_0}(A) \in V^{- 0}(A_2)$$
for all $s \in J_{s_0}$. By inductive hypothesis, we conclude that 
$v^{\mu_0}(A)$ is fixed by the whole 
$$H' :=\begin{bmatrix}\I_m & \\ & \SL(n,\R)\end{bmatrix} \cong \SL(n,\R).$$ In particular, 
$v^{\mu_0}(A)$ is fixed by $A_2$. Let the direct sum 
$$V^{\mu_0}(A) = \bigoplus_{n\delta_1 + \delta_2 = \mu_0} V^{\delta_1,\delta_2}$$
be as in the proof of Lemma \ref{basic_lemma_1}. From the proof of Lemma \ref{basic_lemma_1} we know that any nonzero projection $(v^{\mu_0}(A))^{\delta_1, \delta_2}$ of $v^{\mu_0}(A)$ with respect to this direct sum satisfies $\delta_1$ (the eigenvalue for $A_1$) is non-negative. Because we have $\delta_2 =0$ and $n\delta_1+ \delta_2 \leq 0$, we conclude that 
$\delta_1 = \delta_2 =0$. This implies that $\mu_0 =0$. By Lemma \ref{derivative_lemma}, we have $v^0(A)$ is invariant under $\{u(h \varphi^{(1)}(s_0)): h\in \R\}$. By our assumption, $\varphi^{(1)}(s_0)$ has rank $b$. By conjugating it with elements in $H'$, we have that $u(X)$ fixes $v^{0}(A)$ for any $X$ with rank $b$. Note that the space spanned by all rank $b$ matrices is the whole space $\M(m\times n ,\R)$. This shows that $v^0(A)$ is invariant under the whole $U^{+}(A)$. Since $v^0(A)$ is also invariant under $A$, $v^0(A)$ is invariant under the whole group $H$. Since we assume that $V$ is an irreducible representation of $H$, we conclude that $V$ is trivial.
\par This completes the proof.
\end{proof}
\par Lemma \ref{basic_lemma_2} is sufficient to prove the equidistribution result under the {\em supergeneric\/} condition. Now we consider the case $n=km$ and the curve 
$$\varphi: I=[a,b] \rightarrow \M(m \times n, \R)$$
is {\em generic\/} but not {\em supergeneric\/}. In this case, we will prove the following result, which can be thought of as a  generalization of Corollary~\ref{cor:sl2}. 


\begin{lemma}
\label{basic_lemma_3}
Let $n=km$ and 
$$\varphi: I=[a,b]\rightarrow \M(m\times n, \R)$$
be an analytic {\em generic\/} curve. Let $V$ be an irreducible representation of $H=\SL(m+n,\R)$ and $v \in V$ be a nonzero vector of $V$. Let $A$ denote the diagonal subgroup as before. Suppose for all $s \in I$, 
$$u(\varphi(s))v \in V^{-0}(A),$$
then for all $s_0 \in I$ satisfying the {\em generic\/} condition, we have
$$(u(\varphi(s_0))v)^{0}(A) = \xi(s_0) u(\varphi(s_0))v$$
for some $\xi(s_0) = P^-(A) :=  U^{-}(A) Z_H(A) \subset H$.
\end{lemma}

This lemma is crucial for describing the obstruction to equidistribution for {\em generic\/} curves as done in \S \ref{obstruction_equidistribution}.

\begin{definition}
\label{def:standard}
Assume $n=km$, then we could write $\Phi \in \M(m\times n, \R)$ as $[\Phi_1; \Phi_2; \dots; \Phi_k]$ where $\Phi_i$ denotes the $i$-th $m$ by $m$ block of $\Phi$. An analytic curve $\varphi: I=[a,b]\rightarrow \M(m\times n, \R)$ is called {\em standard} at $s_0 \in I$ if there exist $k$ points $ s_1, \dots, s_k \in I$ such that for $i=1,\ldots,k$, we have 
$$
\varphi(s_i)-\varphi(s_0)  =  [\vect{0}; \dots; \varphi_i(s_i)-\varphi_i(s_0); \dots; \vect{0}] ,
$$
where $\varphi_i(s_i)-\varphi_i(s_0)$ is invertible, it appears in the $i$-th $m\times m$ block and all other blocks are $\vect{0}$. 
\end{definition}

In order to prove Lemma~\ref{basic_lemma_3}, we will need the following lemma.

\begin{lemma}
\label{technical_lemma}
Assume $n=km$. For any analytic curve 
$$\varphi: I=[a,b]\rightarrow \M(m\times n, \R)$$
 which is {\em generic} at $s_0 \in I$, there exists an element $z' =z'(s_0) \in Z_H(A)$ depending analytically on $s_0$, such that the conjugated curve 
$$\phi:=z'\cdot \varphi: I=[a,b]\rightarrow \M(m\times n, \R)$$
is {\em standard} at $s_0$; where the action of $Z_H(A)$ on $\M(m\times n, \R)$ is given by \eqref{eq:ZonM}.

\end{lemma}
\begin{proof}
Replacing $\varphi(s)$ by $\varphi(s)-\varphi(s_0)$, we may assume that $\varphi(s_0) = \vect{0}$.
\par We will prove the statement by induction on $k$.
\par When $k=1$, the statement follows from the definition of {\em generic\/} property.
\par Suppose the statement holds for all $k' < k$. Then we will prove the statement for $n=km$.
\par We write $\varphi(s)=[\varphi_1(s);\varphi_2(s);\dots; \varphi_k(s)]$, where $\varphi_i(s)$ is the $i$-th $m$ by $m$ block of $\varphi(s)$. From the definition of {\em generic\/} property (Definition~\ref{def:geometric_condition}), there exist a subinterval $J_{s_0}\subset I$ such that for $s \in J_{s_0}$, 
$\varphi_1(s)$ is invertible, and the curve $\psi: J_{s_0} \rightarrow \M(m \times (n-m), \R)$
defined by $\psi(s) =[\psi_1(s); \psi_2(s);\dots; \psi_{k-1}(s)]$, where 
\[
\psi_i(s)= \varphi^{-1}_1(s) \varphi_i(s)
\]
is {\em generic\/}.
\par As before, let us denote
$$u'(\psi(s)) =\begin{bmatrix}\I_m & & \\ & \I_m & \psi(s) & \\ & & \I_{n-m}\end{bmatrix}\in Z_{H}(A)$$
for $s \in J_{s_0}$. Now we fix a point $s_1 \in J_{s_0}$ and a subinterval $J_{s_1} \subset J_{s_0}$ such that $\psi$ satisfies the {\em generic\/} condition for $s_1$ and $J_{s_1}$. Replacing $\varphi$ by
$u'(\psi(s_1))\cdot \varphi$, we get 
\[
\varphi(s_1) = [\varphi_1(s_1); \vect{0};\dots; \vect{0}] \text{ and } \psi(s_1)=\vect{0}.
\]

Let 
$$A':= \left\{a'(t):= \begin{bmatrix}\I_m & & \\ & e^{(n-m)t}\I_{m} & \\ & & e^{-mt}\I_{n-m}\end{bmatrix}: t \in \R\right\}$$
and 
$$H' := \left\{\begin{bmatrix}\I_m & \\ & X\end{bmatrix}: X \in \SL(n,\R)\right\}\subset Z_{H}(A).$$
By inductive hypothesis, there exists $z'' \in Z_{H'}(A') \subset Z_H(A)$, such that 
$z'' \cdot \psi$ is {\em standard} at $s_1$. Since $\psi(s_1)=\vect{0}$, there exist $s_2, s_3, \dots, s_k \in J_{s_1}$ such that
$$
z'' \cdot \psi(s_i) =  [\vect{0} ; \dots; \psi_{i-1}(s_i); \dots; \vect{0}], \text{ for }i=2,\dots, k,
$$
where the $(i-1)$-th $m\times m$ block $\psi_{i-1}(s_i)$ is invertible. 
Now we replace $\varphi$ by $z'' \cdot \varphi$. Note that by definition, 
$\varphi_i(s) = \varphi_1 (s) \psi_{i-1}(s)$ for $i=2,\dots, k$, and $s \in J_{s_0}$. Thus, we have for $i=2,\dots, k$,
$$\varphi(s_i) = [\varphi_1(s_i); \vect{0}; \dots; \vect{0}; \varphi_1 (s_i) \psi_{i-1}(s_i); \vect{0};\dots; \vect{0}].$$
Let $z_1$ denote the following element:
$$z_1 := \begin{bmatrix}\I_m &  & & & \\ & \I_m & & & \\ & \psi^{-1}_1(s_2) & \I_m & & \\ & \vdots & & \ddots & \\ &  \psi^{-1}_{k-1}(s_k) & \vect{0} & \cdots & \I_m\end{bmatrix} \in Z_H(A).$$
By direct calculation, we have that $z_1 \cdot \varphi$ is {\em standard} at $s_0$ with given $s_1, s_2, \dots, s_k$.
\par This completes the proof.
 
\end{proof}
\par Now we are ready to prove Lemma \ref{basic_lemma_3}.
\begin{proof}[Proof of Lemma \ref{basic_lemma_3}]
By Lemma \ref{technical_lemma}, we may conjugate the curve by some $z'(s_0) \in Z_H(A)$, such that the conjugated curve, which we still denote by $\varphi$, satisfies the following: there exist $s_1, s_2, \dots, s_k \in I$, such that
$$\varphi(s_i) -\varphi(s_0) = [\vect{0} ;\dots; \varphi_i(s_i)-\varphi_i(s_{0}) ; \vect{0}; \dots; \vect{0}] \text{ for } i=1,2,\dots, k.$$
Replacing $v$ by $u(\varphi(s_0))v$ and $\varphi(s)$ by $\varphi(s)-\varphi(s_0)$, we may assume that $\varphi(s_0) = \vect{0}$ and $v \in V^{-0}(A)$. Then it suffices to show that 
\begin{equation} \label{eq:U-}
v = \xi v^0(A)\text{, for some }\xi \in P^{-}(A).
\end{equation}

For each $i=1,2,\dots, k$, let 
$$A_i : = \left\{a_i(t) := \begin{bmatrix}e^t \I_m & & & & \\ & \ddots & & &  \\ & & e^{-t}\I_m & & \\ & & & \ddots & \\ & & & & \I_m\end{bmatrix}: t \in \R\right\},$$
where $e^{-t}\I_m$ appears in the $(i+1)$-th $m\times m$ diagonal block. We denote its Lie algebra by 
$$\mathfrak{a}_i := \{t \mathcal{A}_i : t \in \R \},$$
where $\mathcal{A}_i := \log a_i(1)$.
Let $\SL(2, \varphi(s_i))$ denote the $\SL(2,\R)$ copy in $H$ containing $A_i$ as the diagonal subgroup and $\{u(r \varphi(s_i)): r \in \R\}$ as the upper triangular unipotent subgroup. 
\par We express the representation $V$ as the direct sum of common eigenspaces of $A_1, A_2, \dots, A_k$:
\begin{equation} \label{eq:Vdelta}
V = \bigoplus_{\vect{\delta}=(\delta_1, \dots , \delta_k)} V(\vect{\delta}),
\end{equation}
where $$V(\vect{\delta}) :=\left\{v \in V : a_i(t) v = e^{\delta_i t} v 
\text{ for all } i = 1,2, \dots, k \text{ and } t \in \R\right\}.$$

\par Let  $w \in V(\vect{\delta})$. We claim that 
for all $ i =1,2, \dots, k$ and $\vect{e}_i=(-1,\ldots,-2,\ldots,-1)$, with $2$ in the $i$-th coordinate, 
\begin{equation} \label{eq:nw}
\mathfrak{n}(\varphi(s_i))w \in V(\vect{\delta}-\vect{e}_i),
\end{equation}
recall that $\mathfrak{n}(\varphi(s_i)) = \log u(\varphi(s_i))$.
\par It is straight forward to check that 
\[
[\mathcal{A}_i, \mathfrak{n}(\varphi(s_i))] = 2 \mathfrak{n}(\varphi(s_i)) \text{ and } 
[\mathcal{A}_j, \mathfrak{n}(\varphi(s_i))] =  \mathfrak{n}(\varphi(s_i)) \text{ for $j\neq i$}.
\]
Therefore, 
\[
\mathcal{A}_j \mathfrak{n}(\varphi(s_i)) w 
 = \mathfrak{n}(\varphi(s_i)) \mathcal{A}_j w + [\mathcal{A}_j, \mathfrak{n}(\varphi(s_i))] w  
= \begin{cases} (\delta_j+1) w & \text{if } j\neq i \\
                          (\delta_i+2) w & \text{if } j=i.  
    \end{cases}
    \]
This proves \eqref{eq:nw}.

\par Let $\mathcal{A} := \log a(1)$, it is easy to see that $\mathcal{A} = \mathcal{A}_1 + \cdots + \mathcal{A}_k$. Therefore,
\[
V^{\sigma}(A) = \bigoplus_{\delta_1 + \cdots + \delta_k = \sigma} V(\delta_1, \dots, \delta_k).
\]

\par On the other hand, because $A_1, \dots, A_k$ normalize $\SL(2, \varphi(s_i))$ for any $i=1,\dots,k$, we can decompose $V$ into the direct sum of irreducible representations $V_{p}$ of $\SL(2,\varphi(s_i))$ which are invariant under $A_1,\dots,A_k$: 
\begin{equation} \label{eq:Vp}
V = \bigoplus_{p} V_{p}.
\end{equation}
For each $V_{p}$,  we can choose a basis $\{w_0, w_1, \dots, w_l\}$, called a {\em  standard basis\/}, of $V_{p}$ such that for each $1\leq r\leq l$,  $w_{r}$ is contained in some weight space $V(\delta_1, \delta_2,\cdots, \delta_k)$, and we index the basis elements such that $a_i(t) w_r = e^{(l-2r)t} w_r$; that is,
\begin{equation} \label{eq:wr}
\text{if }
w_r \in V(\delta_1, \delta_2, \cdots, \delta_k) \text{ then } \delta_i = l -2r.
\end{equation}
Moreover since $\mathfrak{n}(\varphi(s_i))w_s$ is a nonzero multiple of $w_{s-1}$ for $1\leq s\leq l$, by \eqref{eq:nw} we have that  
\begin{equation} \label{eq:rj}
w_{r-j}\in V(\vect{\delta}-j\vect{e_i}), \quad \text{for } r-l\leq j\leq r. 
\end{equation}

\par Let 
$$\pi_p: V \rightarrow V_p $$
denote the canonical projection from $V$ to $V_p$ with respect to \eqref{eq:Vp}, and let 
$$q(\vect{\delta}): V \rightarrow V(\vect{\delta})$$
denote the canonical projection from $V$ to 
$V(\vect{\delta})$ with respect to \eqref{eq:Vdelta}. Then
\begin{equation} \label{eq:piq}
\pi_{p}\circ q(\vect{\delta})=q(\vect{\delta})\circ \pi_{p}.
\end{equation}

\par  We call a vector 
$\vect{\delta}\in \Z^k$ {\em admissible\/} if it can be written as $c_1 \vect{e}_1 + c_2 \vect{e}_2 + \cdots + c_k \vect{e}_k$, where $c_1, c_2, \dots, c_k$ are non-negative integers. 

\begin{claim}
\label{claim_1}
Any $\vect{\delta}=(\delta_1,\ldots,\delta_k)$ such that $q(\vect{\delta})(v)\neq 0 $ is {\em admissible}.
\end{claim}

\begin{proof}[Proof of Claim \ref{claim_1}] By \eqref{eq:wr}, $\delta_i \in \Z$ for $i=1,2,\dots, k$.  For $\vect{\delta}=(\delta_1,\ldots,\delta_k)$, define $\sigma(\vect{\delta}):=\delta_1+\cdots+\delta_k\in \Z$. Since $v\in V^{0-}(A)$, we have $\sigma(\vect{\delta})\leq 0$. We now begin by assuming that the statement of this claim is valid for any $\vect{\delta}'$ such that $\sigma(\vect{\delta}')>\sigma(\vect{\delta})$; note that the statement is vacuously true if $\sigma(\vect{\delta})=0$ (in fact, in this case we have that $\vect{\delta} = \vect{0}$). 

Let  $1\leq i\leq k$ be such that $\delta_i=\min(\delta_1,\dots,\delta_k)$. Then
\begin{equation} \label{eq:delta_i}
\text{$\delta_{i}\leq \sigma(\vect{\delta})/k\leq 0$, and if $\delta_{i}=0$ then $\vect{\delta}=\vect{0}$}. 
\end{equation}
For this choice of $i$, consider the decomposition \eqref{eq:Vp} of $V$ as $V = \bigoplus_{p} V_p$ with respect to the action of $\SL(2,\varphi(s_{i}))$. There exists some $V_p$ such that $\pi_{p}(q(\vect{\delta})v)\neq\vect{0}$. If $\{w_0, w_1 ,\dots, w_l\}$ denotes the standard basis of $V_{p}$, then by \eqref{eq:wr}, $\pi_{p}(q(\vect{\delta})v)$ is a nonzero multiple of $w_{r}$ for some $0\leq r\leq l$ such that 
$\delta_i=l-2r$.

If $\pi_p(v)$ has a non-zero coefficient on $w_{r-j}$ for some $1\leq j\leq r$, then by \eqref{eq:rj}, we have $w_{r-j} \in V(\vect{\delta}-j\vect{e}_i)$.  But then $q(\vect{\delta}-j\vect{e}_i)(v)\neq 0$ and $\sigma(\vect{\delta}-j\vect{e}_i)=\sigma(\vect{\delta})+j(k+1)>\sigma(\vect{\delta})$. By our inductive hypothesis, $\vect{\delta}-j\vect{e}_i$ is admissible, and hence $\vect{\delta}$ is admissible.

Now we can suppose that $\pi_p(v)$ is contained in the span of $w_{r},\ldots,w_{l}$. Then 
\[
a_i(t)w_{r+j}=e^{(\delta_i-2j)t}w_{{r+j}} \text{ and }\delta_{i}-2j\leq \delta_{i}, \quad \forall \,j=0,\dots,l-r. 
\]
Therefore by \eqref{eq:max-urv} in Lemma~\ref{yang_lemma} applied to $V_p$ and the action of $\SL(2,\varphi_i(s_{i}))$, we have that 
$\pi_p(u(\varphi_i(s_i))v)=u(\varphi_i(s_i))\pi_p(v)$ has a nonzero coefficient on $w_{r-j}$ for some $j$ such that 
\[
a_i(t)w_{r-j}=e^{(\delta_{i}+2j)t}w_{r-j} \text{ and } \delta_{i}+2j\geq -\delta_{i}.
\]
Hence $j\geq -\delta_{i}$. By \eqref{eq:rj} $w_{r-j}\in V(\vect{\delta}-j\vect{e}_{i})$. Therefore,
\[
q(\vect{\delta}-j\vect{e}_{i})(\pi_{p}(u(\varphi_i(s_i))v)\neq 0.
\]
By \eqref{eq:delta_i},
\begin{equation} \label{eq:sigdel}
\sigma(\vect{\delta}-j\vect{e}_{i})=\sigma(\vect{\delta})+j(k+1)\geq 
\sigma(\vect{\delta})-(k+1)\delta_{i} \geq \sigma(\vect{\delta})(1-(k+1)/k)\geq 0.
\end{equation}
By our assumption, $u(\varphi_i(s_i))v\in V^{0-}(A)$, and hence $\sigma(\vect{\delta}-j\vect{e}_{i})\leq 0$. Therefore all terms in \eqref{eq:sigdel} are zero. Therefore $\sigma(\vect{\delta})=0$ and $\delta_{i}=0$. Therefore by \eqref{eq:delta_i}, we have that $\vect{\delta}=\vect{0}$, which is admissible. This completes the proof of Claim~\ref{claim_1}.
\end{proof}

\par Now we get back to the proof of \eqref{eq:U-}. For $i=0,1,\dots, k$, let us denote 
\[
E_{0}=\{\vect{0}\} \text{ and } E_i := \{ c_1 \vect{e}_1 + \cdots + c_i \vect{e}_i : c_1, \dots, c_i \in \Z_{\geq 0}\},
\]
and define for any $v'\in V$, 
\begin{equation} \label{eq:vsubi}
v'_i := \sum_{\vect{\delta} \in E_i} q(\vect{\delta}) (v').
\end{equation}

By Claim~\ref{claim_1}, $v=v_{k}$ and $v^{0}=q(\vect{0})(v)=v_{0}$. Therefore, in order to prove \eqref{eq:U-}, it is sufficient to show the following:

\begin{equation}  \label{eq:vi}
v_{i}\in U^{-}(A) v_{i-1}, \quad \text{for all } 1\leq i\leq k.
\end{equation}

To prove this, fix any $1\leq i\leq k$ and consider the decomposition
$$V = \bigoplus_{p} V_p$$
as in \eqref{eq:Vp} into $\SL(2,\varphi(s_{i}))$-irreducible and $A_{1},\dots,A_{k}$-invariant subspaces $V_{p}$. Let $\pi_p : V \rightarrow V_p$ denote the canonical projection with respect to this decomposition. By \eqref{eq:piq} and \eqref{eq:vsubi},
\[
\pi_{p}(v')_{j}=\pi_{p}(v'_{j}), \quad\text{for all $v'\in V$ and $j=0,1,\dots,k$}. 
\]
Hence
\[
\pi_{p}(v_{i-1})=\pi_{p}((v_{i})_{i-1})=\pi_{p}(v_{i})_{i-1}.
\]
Therefore 
\begin{equation} \label{eq:p0}
\text{if $\pi_{p}(v_{i})=0$ then $\pi_{p}(v_{i-1})=0$.}
\end{equation}

Now suppose that $\pi_{p}(v_{i})\neq 0$. Let $\{w_{0},\dots,w_{l}\}$ denote a standard basis of $V_{p}$; that is, \eqref{eq:wr} holds. Let $0\leq r\leq l$ be such that 
 \begin{equation} \label{eq:viwr}
 \pi_{p}(v_{i})\subset \Span\{w_{r},\dots,w_{l}\}\setminus \Span\{w_{r+1},\dots,w_{l}\}.
 \end{equation}
 In particular, $\pi_{p}(v_{i})$ has a nonzero projection on $w_{r}$. Hence by Claim~\ref{claim_1},
 \begin{equation} \label{eq:wrci}
 w_{r}\in V(c_{1}\vect{e}_{1}+\cdots+c_{i}\vect{e}_{i}) \quad \text{for some $c_{1},\dots,c_{i}\in \Z_{\geq0}$.}
 \end{equation}
 By \eqref{eq:rj} we have that
 \begin{equation}
 \label{eq:r+j}
w_{r+j}\in V(c_{1}\vect{e}_{1}+\dots+c_{i-1}\vect{e}_{i-1}+(c_{i}+j)\vect{e}_{i}), \
 \quad \text{for all } -r\leq j\leq l-r.
 \end{equation}
 Therefore $\pi_{p}(v)\in \sum_{\vect{\delta}\in E_{i}} V^{\vect{\delta}}$. Hence
 \begin{equation} \label{eq:viv}
 \pi_{p}(v_{i})=\pi_{p}(v)_{i}=\pi_{p}(v).
 \end{equation}
 By \eqref{eq:r+j} we have
 \begin{equation} \label{eq:wrj}
 a_{i}(t)w_{r+j}=e^{-(\lambda+2j) t} \quad  \text{for $-r\leq j\leq l-r$, where $\lambda=c_{1}+\dots+c_{i-1}+2c_{i}$}.
 \end{equation}

 We apply Lemma~\ref{yang_lemma} to the $\SL_{2}(\varphi(s_{i}))$-action on $V_{p}$ and the vector $\pi_{p}(v_{i})$.  Let  $-r\leq j\leq l-r$ be such that 
 \begin{equation} \label{eq:sir}
 u(\varphi(s_{i}))\pi_{p}(v_{i})\subset \Span\{w_{r+j},\dots,w_{l}\} \setminus \Span\{w(r+j+1),\dots,w_{l}\}.
 \end{equation}
 Then by \eqref{eq:max-urv}, \eqref{eq:viwr} and  \eqref{eq:wrj} we get 
 \begin{equation} \label{eq:j-lambda}
-(\lambda+2j)\geq \lambda.
\end{equation}
On the other hand by \eqref{eq:viv} and our basic assumption we have 
\[
u(\varphi(s_{i}))\pi_{p}(v_{i})=u(\varphi(s_{i}))\pi_{p}(v)=\pi_{p}(u(\varphi(s_{i}))v)\in V^{0-}(A).
\]
Hence by \eqref{eq:r+j} and \eqref{eq:sir} we have
\begin{equation} \label{eq:lambda0-}
0\geq \sigma(c_{1}\vect{e}_{1}+\cdots+c_{i-1}\vect{e}_{i-1}+(c_{i}+j)\vect{e}_{i})=(\lambda-c_{i}+j)(-(k+1)).
 \end{equation}
 Now combining \eqref{eq:j-lambda} and \eqref{eq:lambda0-}, and  we get 
\[
c_{i}\leq \lambda+j \leq 0. 
\]
On the other hand, by \eqref{eq:wrci}, $c_{i}\geq 0$. Therefore $c_{i}=0$ and $j=-\lambda$.  Since $c_{i}=0$, by \eqref{eq:r+j} we have that the projection of $\pi_{p}(v_{i})$ on the line $\R w_{r}$ equals 
\[
\pi_{p}(v_{i})_{i-1}=\pi_{p}((v_{i})_{i-1})=\pi_{p}(v_{i-1}).
\]
 And since $j=-\lambda$, we have equality in  \eqref{eq:j-lambda}, which corresponds to equality in \eqref{eq:max-urv} of Lemma~\ref{yang_lemma}. Therefore \eqref{eq:equality} holds and in view of Definition~\ref{SL2X}, we get
 \begin{equation} \label{eq:pvi}
 \pi_{p}(v_{i})=u^{-}(\vect{0},\dots,-\varphi(s_{i})^{-1},\dots,\vect{0})\pi_{p}(v_{i-1})=\pi_{p}(u^{-}(0,\dots,-\varphi(s_{i})^{-1},\dots,0)v_{i-1}),
 \end{equation}
 where 
\[
u^{-}(\vect{0},\dots,-\varphi(s_{i})^{-1},\dots,\vect{0}):=
\left[
\begin{matrix}
I_{m}                     &           &            &              &               & \\
\vect{0}                 & I_{m} &              &              &             &  \\
\vdots                   &           & \ddots  &              &              &  \\
-\varphi(s_{i})^{-1}&           &              & I_{m}   &               &  \\
\vdots                   &           &               &            &   \ddots   & \\
\vect{0}                 &           &              &             &                 &  I_{m}
\end{matrix}
\right]
 \in \SL(2,\varphi(s_{i}))\cap U^{-}(A). 
 \] 
 Therefore, due to \eqref{eq:p0}, \eqref{eq:pvi} holds for all $p$, and hence \eqref{eq:vi} holds.
 \par This completes the proof.
 \end{proof}

\begin{remark}
\label{rmk:basic-lemma-3}
$\quad$
\begin{enumerate}
\item Though our proof works for the special case $n=km$, we conjecture that the conclusion of Lemma~\ref{basic_lemma_3} should hold for the case of general $(m,n)$.

\item From the proof we can see, if we assume $\varphi(s_0) = \vect{0}$ and $v \in V^{-0}(A)$, then $z'(s_0) \cdot v^0(A)$ is fixed by 
$$B : = \left\{b(t_1 , t_2, \dots, t_k) := \begin{bmatrix}e^{t_1} \I_m & & &  \\ & e^{t_2}\I_m & &  \\ & & \ddots & \\ & & & e^{t_k}\I_m\end{bmatrix}: t_1 + t_2 + \cdots + t_k =0\right\},$$
which is the diagonal subgroup generated by $A_1 , A_2, \dots , A_k$.
\end{enumerate}
\end{remark}


\section{Proof of the equidistribution result}
\label{proof_equidistribution}
In this section we will prove Theorem \ref{goal_thm}. The non-divergence part of the theorem has been proved in \S \ref{non-divergence} (see Remark \ref{rmk:non-divergence}). Here we will prove the equidistribution part. The proof is based on Proposition \ref{prop_algebraic_condition} and Lemma \ref{basic_lemma_2}. 
\begin{proof}[Proof of Theorem \ref{goal_thm}]
Suppose $\varphi: I \rightarrow \M(m\times n, \R)$ is {\em supergeneric\/}, and the normalized parametric measures $\{\lambda_t: t>0\}$ do not tend to the Haar measure $\mu_G$ along some subsequence $t_i \rightarrow +\infty$. By Proposition \ref{prop_algebraic_condition}, there exists some $L \in \mathcal{L}$ and $\gamma \in \Gamma$ such that
$$u(\varphi(s))g \gamma p_L \in V^{-0}(A)$$
for all $s \in I$. Then by Lemma \ref{basic_lemma_2}, we have that $ v := g\gamma p_L$ is fixed by the whole group $H$. Hence $p_L$
is fixed by the action of
$\gamma^{-1}g^{-1}Hg\gamma$. Thus
$$
\begin{array}{rcl}
\Gamma p_L & = & \overline{\Gamma p_L} \text{ since } \Gamma p_L
 \text{ is discrete} \\
 & = & \overline{\Gamma \gamma^{-1} g^{-1}Hg\gamma
 p_L} \\ & = & \overline{\Gamma g^{-1}Hg\gamma
 p_L} \\ & = & G g\gamma p_L \text{ since } \overline{Hg
 \Gamma}=G \\
  & = & G p_L.
 \end{array}
$$
This implies $G_0 p_L = p_L$ where $G_0$ is the connected component
of $e$. In particular, $\gamma^{-1} g^{-1}H g
\gamma\subset G_0$ and $G_0 \subset N^1_{G}(L)$. By \cite[Theorem
2.3]{Shah_3}, there exists a closed subgroup $F_1 \subset
N^1_{G}(L)$ containing all $\mathrm{Ad}$-unipotent one-parameter
subgroups of $G$ contained in $N^1_{G}(L)$ such that $F_1 \cap
\Gamma$ is a lattice in $F_1$ and $\pi (F_1)$ is closed. If we put
$F= g\gamma F_1 \gamma^{-1} g^{-1}$, then $H\subset
F$ since $H$ is generated by it unipotent
one-parameter subgroups. Moreover, $Fx =g \gamma \pi(F_1)$ is closed
and admits a finite $F$-invariant measure. Then since
$\overline{Hx}=G/\Gamma$, we have $F=G$. This implies
$F_1 = G$ and thus $L \lhd G$. Therefore $N(L,W)=G$. In particular,
$W\subset L$, and thus $L\cap H$ is a normal subgroup
of $H$ containing $W$. Since $H$ is a
simple group, we have $H \subset L$. Since $L$ is a
normal subgroup of $G$ and $\pi(L)$ is a closed orbit with finite
$L$-invariant measure, every orbit of $L$ on $G/\Gamma$ is also
closed and admits a finite $L$-invariant measure, in particular,
$Lx$ is closed. But since $Hx$ is dense in
$G/\Gamma$, $Lx$ is also dense. This shows that $L=G$, which
contradicts our hypothesis that the limit measure is not $\mu_G$.
\par This completes the proof.
\end{proof}


\section{Obstruction of equidistribution}
\label{obstruction_equidistribution}

We will study the obstruction of equidistribution of the expanding curves $\{a(t)u(\varphi(I))x : t>0\}$ as $t \rightarrow +\infty$ and describe limit measures if equidistribution fails. 
\par At first, if the {\em generic\/} condition fails, then it is possible that the limit measure $\mu_{\infty}$ of $\{\lambda_t : t>0\}$ along some subsequence
$t_i \rightarrow \infty$ is not a probability measure. In other words, part of the expanding curves might escape to infinity along this subsequence. In fact, for $G= H = \SL(m+n,\R)$ and $\Gamma = \SL(m+n,\Z)$, it is not hard to construct some special curve $\varphi$ such that $a(t)u(\varphi(I))[e] \rightarrow \infty$ as $t \rightarrow \infty$. We refer the reader to \cite{Klein_Weiss} and \cite{Klein_Mar_Wang} for more examples. In this paper, we focus on {\em generic} curves.
\par If $m$ and $n$ are co-prime, the {\em generic\/} condition is the same as the {\em supergeneric\/} condition, so there is nothing to discuss in this case.
\par Therefore we consider the case $(m,n)>1$ and the analytic curve 
$$\varphi: I =[a,b] \rightarrow \M(m\times n, \R)$$
is {\em generic\/} but not {\em supergeneric\/}. However, for now, we could only handle the case $n = km$ where 
$k>1$ is some positive integer (to handle general $(m,n)$, we need some general version of Lemma \ref{basic_lemma_3}). With these assumptions, we want to describe the obstruction of equidistribution of $\{a(t)u(\varphi(I))x : t >0\}$ as $t \rightarrow \infty$.
\par In this section, we always assume that $n=km$ and the analytic curve 
$$\varphi:I = [a,b] \rightarrow \M(m\times n, \R)$$
is {\em generic\/}.
\par First note that for $L \in \mathcal{L}$, the stabilizer of $p_L$ is $N_G^1(L)$ where 
$$N_G^1(L):= \{g \in G : g L g^{-1} =L \text{ and }  \mathrm{det}(\mathrm{Ad}(g)|_{\mathfrak{l}}) =1 \}.$$


\begin{theorem}[See ~{\cite[Proposition 4.9]{Shah_1}}]
\label{group_level_thm}
Suppose the expanding curves $a(t)u(\varphi(s))x$ are not tending to be equidistributed along some subsequence $t_i \rightarrow +\infty$. By Proposition \ref{prop_algebraic_condition}, there exist $L \in \mathcal{L}$ and $\gamma \in \Gamma$ such that 
$$u(\varphi(s))g\gamma p_L \in V^{-0}(A),$$
for all $s \in I$. Then there exist $h \in P^- (A)$ and some algebraic subgroup $F$ of $G$ containing $A$ such that $h^{-1} F h $ fixes $v := g \gamma p_L$ and 
$$u(\varphi(s)) \in P^{-}(A) F h$$
for all $s \in I$. Recall that $P^{-}(A) = U^{-}(A)Z_H(A)$ denotes the maximal parabolic subgroup of $H$ associated with $A$.
\end{theorem}

\begin{proof}
Let $s_0\in I$ such that every point in a neighborhood $J$ of $s_0$ satisfies the {\em generic} condition. Replacing $v$ by $u(\varphi(s_0))v$ and $\varphi(s)$ by $\varphi(s)-\varphi(s_0)$, we may assume that $\varphi(s_0) = \vect{0}$. By Lemma \ref{basic_lemma_3}, for every $s \in I$, we have that
$$\lim_{t \rightarrow \infty} a(t)u(\varphi(s)) v = (u(\varphi(s)) v)^{0}(A) = \xi(s) u(\varphi(s))v, $$
for some $\xi(s)\in P^-(A)$. Let $p_0=v^0(A)$. Then $p_0=\xi(s_0)v$. This implies that 
$$\lim_{t \rightarrow +\infty} a(t)u(\varphi(s)) \xi(s_0)^{-1} p_0 = \xi(s) u(\varphi(s)) \xi(s_0)^{-1} p_0.$$

Let $F_1 : = N_G^1(L)$, then $F_1$ is the stabilizer of $p_L = (g \gamma)^{-1} v=(g\gamma)\inv\xi(s_0)\inv p_0$. Let $F := (\xi(s_0) g \gamma) F_1 (\xi(s_0) g \gamma)^{-1}$, then
$F$ is the stabilizer of $p_0$. Then $A \subset F$ since $p_0$ is invariant under $A$.



\par Since $Gp_0$ is open in its closure, the map 
$g F \mapsto g p_0 : G/F \rightarrow G p_0 $ is a homeomorphism. Thus we have that in $G/F$,
\begin{equation} \label{eq:F}
\lim_{t \rightarrow +\infty} a(t)  u(\varphi(s)) \xi(s_0)\inv F = \xi(s) u(\varphi(s)) \xi(s_0)\inv F.
\end{equation}

Since the Lie algebra of $F$ is $\{\Ad(a(t)): t \in \R \}$-invariant, there exists an $\{\Ad(a(t)): t \in \R \}$-invariant subspace $W$ of the Lie algebra of $H$ complementary to the Lie algebra of $F$. We decompose $W=S^0\oplus W^-\oplus W^+$ into the fixed point space, the contracting subspace and the expanding subspace for the action of $\Ad(a(t))$ as $t\to + \infty$. Then for all $s\in J$ near $s_0$ we have, 
\begin{equation} \label{eq:W}
\xi(s) u(\varphi(s)) \xi(s_0)^{-1} F = \exp(w^0(s))\exp(w^-(s))\exp(w^+(s))F,
\end{equation}
for all $s\in J$ near $s_0$, where $w^0\in W^0$, and $w^\pm\in W^\pm$. Combing \eqref{eq:F} and \eqref{eq:W}, we deduce that $w^+(s)=0$ for all $s$ near $s_0$. Thus we get $\xi(s)u(\varphi(s)) \xi(s_0)^{-1}F=\exp(w^0(s))\exp(w^-(s))F$. Let $\eta(s):=\exp(w^0(s))\exp(w^-(s))$. It is easy to see that $\eta(s) \in P^-(A)$. Hence
for all $s\in J$ near $s_0$, we have $u(\varphi(s))\in \xi(s)^{-1}\eta(s)F\xi(s_0)$. Therefore by the analyticity of $\varphi$, we get that
\[
\{u(\varphi(s)):s\in I\}\subset P^-(A)F\xi(s_0) \cap U^+(A)=P^-(A)(F\cap H) \xi(s_0) \cap U^+(A).
\]
\par This proves the statement for $h = \xi(s_0)$.
\end{proof}

\begin{remark}  \label{rem:obstruction}
It may be noted that since by a result of Dani and Margulis, the orbit $\Gamma p_L$ is discrete, we have that $\Gamma F_1$ is closed in $G$. Therefore for $h=\xi(s_0)\in P^-(A)$, we have $Fhg\Gamma=(h g \gamma) F_1 \Gamma$ is closed. Thus there exist analytic curves $\xi^- : I \rightarrow U^-(A)$ and $\xi^0: I \rightarrow Z_H(A)$ such that $u(\varphi(s))x\subset \xi^-(s)\xi^0(s)Fh x$ for all $s\in I$. Then as $t\to\infty$, the distance between $a(t)u(\varphi(s))x$ and $\xi^0(s)Fh x$ tends to zero, since $Fh x=Fhg\Gamma$ is a proper closed $\{a(t)\}$-invariant subset of $G/\Gamma$. Thus every limit measure of the sequence $\{\mu_t\}$ as $t\to +\infty$ is a probability measure whose support is contained in $\xi^0(I)Fhx$. Replacing $F$ by a smaller subgroup containing $A$, we can actually ensure that $Fhx$ admits a finite $F$-invariant measure. 
\end{remark}



We conjecture that in the general case of $(m,n)>1$, if $\varphi$ is {\em generic}, then Lemma~\ref{basic_lemma_3}, Theorem~\ref{group_level_thm}, and Remark~\ref{rem:obstruction} should hold.


\appendix
\section{Relation between the condition given in {\rm \cite{ABRS}} and the {\em generic\/} condition} \label{App:AppendixA}
\par It is worth explaining the condition given in \cite{ABRS} and its relation to our 
{\em generic\/} condition. 
\par We denote $M(s): = [\I_m; \varphi(s) ] \in \M(m \times (m+n),\R)$. Given a subspace $W$ and $0<r < \frac{m \dim W}{m+n}$, we define 
the pencil $\mathcal{P}_{W,r}$ to be 
$$\mathcal{P}_{W,r} := \{M \in \M(m\times (m+n),\R) : \dim M W =r\}.$$
In \cite{ABRS}, the following theorem is announced: if a submanifold is not contained in any such pencil,
then the submanifold is extremal. In our case, it says that if the curve $\{M(s): s \in I\}$ is not contained in any pencil 
$\mathcal{P}_{W,r}$, then the curve is extremal. It is easy to see that if $W$ is a rational subspace, then $\mathcal{P}_{W,r}$ is not extremal.
So this condition is considered almost optimal.
\begin{proposition}
Suppose that the analytic curve $\varphi: I=[a,b] \rightarrow \M(m\times n, \R)$ is {\em generic\/}, then the curve $\{M(s) = [\I_m; \varphi(s)]: s \in I\}$ is not contained in any pencil $\mathcal{P}_{W,r}$.
\end{proposition}
\begin{proof}
Without loss of generality, we may assume that every point in $I$ satisfies the {\em generic} condition.
\par We will prove the statement by induction on $(m,n)$. Without loss of generality, we may assume that $m \leq n$.
\par We first prove the statement holds for $(n,n)$. For contradiction, suppose that there exists some subspace $W$ and $0<r< \frac{\dim W}{2}$ such that
$$M(s)= [\I_n ; \varphi(s)] \in \mathcal{P}_{W,r} \text{ for all } s \in I .$$
This implies that $\mathrm{Ker} M(s) \cap W  > \frac{\dim W}{ 2} $ for all $s \in I$.  Then for any $s_1, s_2 \in I$, the dimension of 
$\mathrm{Ker} M(s_1) \cap \mathrm{Ker} M(s_2) \cap W$ is greater than $0$, since the sum of $\dim ( \mathrm{Ker} M(s_1) \cap W)$
and $\dim ( \mathrm{Ker} M(s_2) \cap W)$ is greater than $\dim W$. It is easy to see that
 $$\mathrm{Ker} M(s) = \{(-\varphi(s)w , w ) : w \in \R^n\}.$$
 Therefore, there exist $w_1, w_2 \in \R^n\setminus\{\vect{0}\}$ such that $(-\varphi(s_1)w_1, w_1) = (-\varphi(s_2) w_2, w_2)$. This implies $w_1 = w_2$ and 
 $\varphi(s_1)w_1 = \varphi(s_2) w_1$. Therefore $(\varphi(s_1) - \varphi(s_2)) w_1 =0$. But this is impossible since $w_1 \neq \vect{0}$ and 
 $\varphi(s_1) - \varphi(s_2)$ is invertible. This contradiction shows the statement for $(n,n)$.
 \par Suppose the statement holds for all $(m',n')$ such that $m' \leq m$, $n'\leq n$ and $m'+ n' < m+n$, we want to prove the statement for 
 $(m,n)$. Suppose not, then the curve $\{M(s) = [\I_m; \varphi(s)]: s \in I\}$ is contained in some pencil $\mathcal{P}_{W,r}$ where $r < \frac{m \dim W }{m+n}$. Let us fix some $s_0 \in I$ and denote $W_0 := \mathrm{Ker} M(s_0) \cap W $. Then from our assumption we have that $\dim W_0 = \dim W -r$. For any $s \in I$, since $\dim (\mathrm{Ker} M(s)  \cap W) = \dim W -r$, we have that 
 $$\dim (\mathrm{Ker}M(s) \cap W_0) = \dim (\mathrm{Ker} M(s) \cap \mathrm{Ker} M(s_0) \cap W) \geq 2(\dim W - r) - \dim W = \dim W -2r = \dim W_0 -r. $$
 Therefore, $\dim M(s) W_0 \leq r$ for all $s \in I$.
 \par We write any $w \in W \subset \R^{m+n}$ as $(w_1, w_2)$ where $w_1 \in \R^m$ and $w_2 \in \R^n$. Since $W_0 = \mathrm{Ker} M(s_0) \cap W$, every $(w_1 , w_2) \in W_0$ satisfies that $w_1 = - \varphi(s_0) w_2$. By identifying $(-\varphi(s_0) w_2 , w_2)$ with $w_2 \in \R^n$, we may consider $W_0$ as a subspace of $\R^n$. By direct calculation, we have that under this identification, $M(s) : W_0 \rightarrow \R^m$ is defined as follows:
  $$M(s): w \in W_0 \mapsto  (\varphi(s)-\varphi(s_0)) w \in \R^m .$$
  Following our previous notation, we may write $\varphi(s) = [\varphi_1 (s); \varphi_2(s)]$ where $\varphi_1(s)$ denotes the first $m$ by $m$ block of 
  $\varphi(s)$ and $\varphi_2(s)$ denotes the rest $m$ by $n-m$ block. By our assumption, $\varphi_1(s) -\varphi_1(s_0)$ is invertible for $s$ inside some subinterval $J_{s_0} \subset I$. Accordingly we may write $w \in W_0 \subset \R^n$ as $(w_3, w_4)$ where $w_3 \in \R^m$ and $w_4 \in \R^{n-m}$. For $s\in J_{s_0}$, let us denote 
  $$\psi(s) : = (\varphi_1 (s) - \varphi_1(s_0))^{-1} (\varphi_2(s) - \varphi_2(s_0)) \in \M(m \times (n - m) , \R)$$ 
  and 
  $N(s) := [\I_m;  \psi(s)] \in \M(m\times n, \R)$. By our assumption, $\psi : J_{s_0} \rightarrow \M(m \times (n-m), \R)$ is {\em generic\/}.
  Then for $w =(w_3 , w_4 ) \in W_0 \subset \R^n$,  
  $$\begin{array}{rcl} M(s) w &= & M(s)(w_3, w_4) \\
  & = & (\varphi_1 (s) - \varphi_1 (s_0)) w_3 + (\varphi_2 (s) - \varphi_2(s_0))w_4  \\
  & = & (\varphi_1(s) - \varphi_1(s_0))( w_3 + \psi(s) w_4)  \\ 
  & = & (\varphi_1 (s) - \varphi_1(s_0)) N(s) (w_3, w_4) .\end{array}$$
Since $\varphi_1 (s) - \varphi_1(s_0)$ is invertible, we have that $\dim M(s) W_0 = \dim N(s) W_0$. Therefore, 
$$\dim N(s) W_0 \leq r , \text{ for all } s \in J_{s_0}.$$ This implies that there exists some $r' \leq r$ and some subinterval $J'_{s_0} \subset J_{s_0}$ such that $\dim N(s) W_0 = r'$ for all $s \in J'_{s_0}$, i.e., 
$$N(s) \in \mathcal{P}_{W_0, r'} \text{ for all }  s \in J'_{s_0} .$$
But this contradicts our inductive assumption for case $(m, n-m)$. In fact, $W_0 \subset \R^{m +(n-m)}$, $N(s) = [\I_m; \psi(s)]$ where the curve $\psi(s) \in \M(m\times (n-m), \R)$ is {\em generic\/}. Thus to apply the inductive assumption for $(m, n-m)$, it suffices to check that $r' < \frac{m \dim W_0}{ n}$. Since $r' \leq r$, we only need to show that $r < \frac{m \dim W_0}{ n}$. The inequality is equivalent to 
$$n r < m \dim W_0 = m (\dim W - r).$$
It is straightforward to check that it is the same as 
$$r < \frac{m \dim W}{ m+n},$$
which is our assumption. This allows us to apply the inductive assumption and conclude the contradiction.
\par This completes the proof.
\end{proof}
Therefore the {\em generic\/} condition implies the condition given in \cite{ABRS}.

\printbibliography

\end{document}